\documentclass{article}%
\usepackage{amssymb}
\usepackage{graphicx}
\usepackage{amsmath}
\usepackage{amsfonts}%
\newtheorem{theorem}{Theorem}

\newtheorem{corollary}[theorem]{Corollary}

\newtheorem{definition}[theorem]{Definition}
\newtheorem{example}[theorem]{Example}

\newtheorem{lemma}[theorem]{Lemma}

\newtheorem{proposition}[theorem]{Proposition}

\newenvironment{proof}[1][Proof]{\textbf{#1.} }{\ \rule{0.5em}{0.5em}}
\begin{document}

\title{The adjacency matroid of a graph}
\author{Robert Brijder\\Hasselt University and Transnational University of Limburg\\Belgium
\and Hendrik\ Jan Hoogeboom\\LIACS, Leiden University\\The Netherlands
\and Lorenzo Traldi\\Lafayette College\\Easton, Pennsylvania 18042 USA}
\date{}
\maketitle

\begin{abstract}
If $G$ is a looped graph, then its adjacency matrix represents a binary
matroid $M_{A}(G)$\ on $V(G)$. $M_{A}(G)$ may be obtained from the
delta-matroid represented by the adjacency matrix of $G$, but $M_{A}(G)$ is
less sensitive to the structure of $G$. Jaeger proved that every binary
matroid is $M_{A}(G)$ for some $G$ [\emph{Ann. Discrete Math.} \textbf{17
}(1983), 371-376].

The relationship between the matroidal structure of $M_{A}(G)$ and the
graphical structure of $G$ has many interesting features. For instance, the
matroid minors $M_{A}(G)-v$ and $M_{A}(G)/v$ are both of the form
$M_{A}(G^{\prime}-v)$ where $G^{\prime}$ may be obtained from $G$ using local
complementation. In addition, matroidal considerations lead to a principal
vertex tripartition, distinct from the principal edge tripartition of
Rosenstiehl and Read [\emph{Ann. Discrete Math.} \textbf{3 }(1978), 195-226].
Several of these results are given two very different proofs, the first
involving linear algebra and the second involving set systems or $\Delta
$-matroids. Also, the Tutte polynomials of the adjacency matroids of $G$ and
its full subgraphs are closely connected to the interlace polynomial of
Arratia, Bollob\'{a}s and Sorkin [\textit{Combinatorica} \textbf{24} (2004), 567-584].

\bigskip


\bigskip

Mathematics Subject\ Classification. 05C50

\end{abstract}

\section{Introduction}

A distinctive feature of matroid theory is that there are so many equivalent
ways to define matroids, each providing its own special insight into the
nature of the structure being defined. We refer to the books of Oxley
\cite{O}, Welsh \cite{We} and White \cite{W1, W2, W} for thorough discussions.
Here is one way to define a particular kind of matroid:

\begin{definition}
\label{binmat}A \emph{binary matroid} $M$ is an ordered pair $(V,\mathcal{C}%
(M))$, which satisfies the following \emph{circuit axioms}:

1. $V$ is a finite set and $\mathcal{C}(M)\subseteq2^{V}$.

2. $\varnothing\notin\mathcal{C}(M)$.

3. If $C_{1},C_{2}\in\mathcal{C}(M)$ then $C_{1}\not \subseteq C_{2}$.

4. If $C_{1}\neq C_{2}\in\mathcal{C}(M)$ then the symmetric difference
$(C_{1}\backslash C_{2})\cup(C_{2}\backslash C_{1})$ $=$ $C_{1}\Delta C_{2}$
contains at least one $C\in\mathcal{C}(M)$.
\end{definition}

If $M$ and $M^{\prime}$ are matroids on $V$ and $V^{\prime}$ then $M\cong
M^{\prime}$ if there is a bijection between $V$ and $V^{\prime}$ under which
$\mathcal{C}(M)$ and $\mathcal{C}(M^{\prime})$ correspond.

We consider $2^{V}$ as a vector space over $GF(2)$ in the usual way: if
$S,S_{1},S_{2}\subseteq V$ then $0\cdot S$ $=$ $\varnothing$, $1\cdot S$ $=$
$S$ and $S_{1}+S_{2}$ $=$ $S_{1}\Delta S_{2}$.

\begin{definition}
If $M$ is a binary matroid on $V$ then the \emph{cycle space} $Z(M)$ is the
subspace of $2^{V}$ spanned by $\mathcal{C}(M)$.
\end{definition}

The importance of the cycle space of a binary matroid is reflected in the well
known fact that two fundamental ideas of matroid theory, nullity and duality,
correspond under $Z$ to two fundamental ideas of linear algebra, dimension and
orthogonality: $\dim Z(M)$ is the nullity of $M$, and if $M^{\ast}$ is the
dual of $M$ then $Z(M^{\ast})$ is the orthogonal complement of $Z(M)$. (See
\cite{O, We, W2} for details.) Ghouila-Houri \cite{GH} showed that the
importance of the cycle space is reflected in another special property,
mentioned by some authors \cite{Bu4, J2, W2} but not stated explicitly in most
accounts of the theory.

\begin{theorem}
\label{bin} \cite{GH} Let $V$ be a finite set. Then the function
\[
Z:\{\text{binary matroids on }V\}\text{ }\rightarrow\text{ }%
\{GF(2)\text{-subspaces of }2^{V}\}
\]
is bijective.
\end{theorem}

Theorem \ref{bin} tells us that any construction or function which assigns a
subspace of $2^{V}$ to some object may be unambiguously reinterpreted as
assigning a binary matroid to that object. There are of course many notions of
linear algebra that involve assigning subspaces to objects. For instance, an
$m\times n$ matrix $A$ with entries in $GF(2)$ has four associated subspaces:
the row space and right nullspace are orthogonal complements in $GF(2)^{n}$,
and the column space and left nullspace are orthogonal complements in
$GF(2)^{m}$. According to Theorem \ref{bin}, we could just as easily say that
an $m\times n$ matrix with entries in $GF(2)$ has four associated binary
matroids, a pair of duals on an $m$-element set, and a pair of duals on an
$n$-element set. For a symmetric matrix the row space and the column space are
the same, and the left and right nullspace are the same.

Let $G$ be a graph. A familiar construction associates to $G$ its
\emph{polygon matroid} $M(G)$, the binary matroid on $E(G)$ whose circuits are
the minimal edge-sets of circuits of $G$. In this paper we discuss a different
way to associate a binary matroid to $G$, which was mentioned by Jaeger in
1983 \cite{J1, J2}; the notion seems to have received little attention in the
intervening decades.

\begin{definition}
\label{mat}Let $G$ be a graph, and let $\mathcal{A}(G)$ be the Boolean
adjacency matrix of $G$, i.e., the $V(G)\times V(G)$ matrix with entries in
$GF(2)$, in which a diagonal entry $a_{vv}$ is 1 if and only if $v$ is looped
and an off-diagonal entry $a_{vw}$ is 1 if and only if $v\neq w$ are adjacent.
Then the \emph{adjacency matroid} $M_{A}(G)$ is the binary matroid on $V$ $=$
$V(G)$ represented by $\mathcal{A}(G)$, i.e., its circuits are the minimal
nonempty subsets $S\subseteq V$ such that the columns of $\mathcal{A}(G)$
corresponding to elements of $S$ are linearly dependent.
\end{definition}

Here are three comments on Definition \ref{mat}:

1. We understand the term \emph{graph} to include multigraphs; that is, we
allow graphs to have loops and parallel edges. Although Definition \ref{mat}
applies to an arbitrary graph $G$, $\mathcal{A}(G)$ does not reflect the
number of edges connecting two adjacent vertices, or the number of loops on a
looped vertex. Consequently, the reader may prefer to think of $\mathcal{A}%
(G)$ and $M_{A}(G)$ as defined only when $G$ is a looped simple graph.

2. In light of Theorem \ref{bin}, $M_{A}(G)$ may be described more simply as
the binary matroid whose cycle space $Z(M_{A}(G))$ is the nullspace of
$\mathcal{A}(G)$.

3. Many graph-theoretic properties of a graph $G$ do not match conveniently
with matroid-theoretic properties of $M_{A}(G)$. For example, recall that a
\emph{loop} in a matroid $M$ is an element $\lambda$ such that $\{\lambda
\}\in\mathcal{C}(M)$, and a \emph{coloop} is an element $\kappa$ such that
$\kappa\notin\gamma$ for all $\gamma\in\mathcal{C}(M)$. In the polygon matroid
of $G$, a loop is an edge incident on only one vertex, and a coloop is a cut
edge. In the adjacency matroid of $G$ we have the following very different
results; (i) tells us that looped vertices of $G$ cannot be loops of
$M_{A}(G)$, and (ii) tells us that coloops of $M_{A}(G)$ cannot in general
have anything to do with connectedness of $G$. \noindent(Result (i) follows
immediately from Definition \ref{mat}, but proving (ii) requires a little more
work; it follows readily from Lemma \ref{noncoloop} below.)

(i) A vertex $v\in V(G)$ is a loop of $M_{A}(G)$ if and only if $v$ is
isolated and not looped in $G$.

(ii) Suppose $v\in V(G)$, and let $G^{\prime}$ be the graph obtained from $G$
by toggling the loop status of $v$. Then $v$ is a coloop of at least one of
the adjacency matroids $M_{A}(G)$, $M_{A}(G^{\prime})$.

Theorem \ref{bin} and the equality $Z(M^{\ast})=Z(M)^{\bot}$ directly imply
the following.

\begin{theorem}
\label{iso}Let $G$ and $G^{\prime}$ be two $n$-vertex graphs, let
$f:V(G)\rightarrow V(G^{\prime})$ be a bijection, and let $2^{f}%
:2^{V(G)}\rightarrow2^{V(G^{\prime})}$ be the isomorphism of $GF(2)$-vector
spaces induced by $f$. Then the following three conditions are equivalent.

1. $f$ defines an isomorphism $M_{A}(G)\cong M_{A}(G^{\prime})$.

2. $2^{f}$ maps the column space of $\mathcal{A}(G)$ onto the column space of
$\mathcal{A}(G^{\prime})$.

3. $2^{f}$ maps the nullspace of $\mathcal{A}(G)$ onto the nullspace of
$\mathcal{A}(G^{\prime})$.

\noindent\noindent Also, the following three conditions are equivalent.

1. $f$ defines an isomorphism $M_{A}(G)\cong M_{A}(G^{\prime})^{\ast}$.

2. $2^{f}$ maps the column space of $\mathcal{A}(G)$ onto the nullspace of
$\mathcal{A}(G^{\prime})$.

3. $2^{f}$ maps the nullspace of $\mathcal{A}(G)$ onto the column space of
$\mathcal{A}(G^{\prime})$.
\end{theorem}

Every graph with at least one edge has the same adjacency matroid as
infinitely many other graphs, obtained by adjoining parallels. Even among
looped simple graphs, there are many examples of nonisomorphic graphs with
isomorphic adjacency matroids. For instance, the simple path of length two has
the same adjacency matroid as the graph that consists of two isolated, looped
vertices. However, a looped simple graph is determined up to isomorphism by
the adjacency matroids of its full subgraphs.

\begin{definition}
Let $G$ be a graph, and suppose $S\subseteq V(G)$. Then $G[S]$ denotes the
\emph{full subgraph of} $G$ \emph{induced by }$S$, i.e., the subgraph with
$V(G[S])$ $=$ $S$ that includes the same incident edges as $G$.
\end{definition}

If $v\in V(G)$ then $G[V(G)\backslash\{v\}]$ is also denoted $G-v$.

\begin{theorem}
\label{isomat}Let $G$ and $G^{\prime}$ be looped simple graphs, and let
$f:V(G)\rightarrow V(G^{\prime})$ be a bijection.\ Then the following are equivalent.

1. \ $f$ is an isomorphism of graphs.

2. For every $S\subseteq V(G)$, $f$ defines an isomorphism $M_{A}(G[S])\cong
M_{A}(G^{\prime}[f(S)])\,$ of matroids.

3. For every $S\subseteq V(G)$ with $\left\vert S\right\vert \leq2$, the
matroids $M_{A}(G[S])$ and $M_{A}(G^{\prime}[f(S)])\,$\ have the same nullity.
\end{theorem}

\begin{proof}
The implications $1\Rightarrow2\Rightarrow3$ are obvious. The implication
$3\Rightarrow1$ follows from these facts: A vertex $v\in V(G)$ is looped
(resp. unlooped) if and only if the nullity of $M_{A}(G[\{v\}])$ is 0 (resp.
1). If $v\neq w\in V(G)$ are both unlooped, then they are adjacent (resp.
nonadjacent) in $G$ if and only if the nullity of $M_{A}(G[\{v,w\}])$ is 0
(resp. 2). If $v\in V(G)$ is looped and $w\in V(G)$ is unlooped, then they are
adjacent (resp. nonadjacent) in $G$ if and only if the nullity of
$M_{A}(G[\{v,w\}])$ is 0 (resp. 1). If $v\neq w\in V(G)$ are both looped, then
they are adjacent (resp. nonadjacent) in $G$ if and only if the nullity of
$M_{A}(G[\{v,w\}])$ is 1 (resp. 0).
\end{proof}

The polygon matroids of graphs constitute a special subclass of the binary
matroids. Jaeger proved that the adjacency matroids of graphs, instead,
include all the binary matroids:

\begin{theorem}
\label{Jae} \cite{J1} Let $M$ be an arbitrary binary matroid. Then there is a
graph $G$ with $M$ $=$ $M_{A}(G)$.
\end{theorem}

Jaeger also gave an original characterization of the polygon matroids of graphs.

\begin{theorem}
\label{Jae2} \cite{J2} The polygon matroids of graphs are the duals of the
adjacency matroids of looped circle graphs.
\end{theorem}

We recall the proofs of Theorems \ref{bin} and \ref{Jae} in Section 2, and in
Sections 3 and 4 we give a new proof of Theorem \ref{Jae2}. This new proof
consists of two parts, which are interesting enough to state separately. (See
Section 3 for definitions.)

\begin{theorem}
\label{corekerequiv}If $P$ is a circuit partition of a 4-regular graph $F$ and
$C$ is an Euler system\ of $F$ that is compatible with $P$, then the polygon
matroid of the touch-graph of $P$ is dual to the adjacency matroid of a
particular looped version of the interlacement graph of $C$.
\end{theorem}

\begin{theorem}
\label{ubiq} Every graph without isolated, unlooped vertices is the
touch-graph of some circuit partition in some 4-regular graph.
\end{theorem}

Theorems \ref{corekerequiv} and \ref{ubiq} are not of only abstract interest.
Touch-graphs of circuit partitions in 4-regular graphs are fairly easy to
understand, and many properties of general adjacency matroids may be motivated
by dualizing properties of touch-graphs. It turns out that the general theory
of adjacency matroids is closely connected to two important notions that also
generalize properties of 4-regular graphs: the $\Delta$-matroids introduced
by\ Bouchet \cite{Bu1, Bu2, Bu3} and the interlace polynomials introduced by
Arratia, Bollob\'{a}s and Sorkin \cite{A1, A2, A}. These close connections are
indicated by the fact that local complementation plays a significant role in
all three theories.

\begin{definition}
Let $G$ be a graph with a vertex $v$. Then the \emph{local complement} $G^{v}$
is the looped simple graph obtained from $G$ by toggling the loop status of
every neighbor of $v$, and toggling the adjacency status of every pair of
neighbors of $v$.
\end{definition}

To be explicit: $G^{v}$ is the looped simple graph related to $G$ as follows:
if $v\neq w\neq x\neq v$ and both $w$ and $x$ are neighbors of $v$ in $G$,
then $G^{v}$ has an edge $wx$\ if and only if $w$ and $x$ are not adjacent in
$G$; if $w\neq x\in V(G)$ and at least one of $w,x$ is not a neighbor of $v$
in $G$, then $G^{v}$ has an edge $wx$\ if and only if $w$ and $x$ are adjacent
in $G$; if $w\neq v$ is a neighbor of $v$ in $G$ then there is a loop on $w$
in $G^{v}$ if and only if $w$ is not looped in $G$; and if $w$ is not a
neighbor of $v$ in $G$ then there is a loop on $w$ in $G^{v}$ if and only if
$w$ is looped in $G$. Note that for every graph $G$, $(G^{v})^{v}$ is the
looped simple graph obtained from $G$ by replacing each set of parallels with
a single edge.

There are two \emph{matroid minor} operations, deletion and contraction.

\begin{definition}
\label{del}If $M$ is a matroid on $V$ and $v\in V$ then the \emph{deletion}
$M-v$ is the matroid on $V\backslash\{v\}$ with $\mathcal{C}(M-v)$ $=$
$\{\gamma\in\mathcal{C}(M)\mid v\notin\gamma\}$.
\end{definition}

If $M$ is a binary matroid, then $M-v$ is the binary matroid on $V\backslash
\{v\}$ with $Z(M-v)=Z(M)\cap2^{V\backslash\{v\}}$.

\begin{definition}
\label{contr}If $M$ is a matroid on $V$ and $v\in V$, then the
\emph{contraction} $M/v$ is the matroid on $V\backslash\{v\}$ with
$\mathcal{C}(M/v)$ $=$ $\{$minimal nonempty subsets $\gamma\subseteq
V\backslash\{v\}\mid\gamma\cup\{v\}$ contains an element of $\mathcal{C}(M)\}$.
\end{definition}

If $M$ is a binary matroid, let $[v]$ denote the subspace of $2^{V}$ spanned
by $\{v\}$; we identify $2^{V\backslash\{v\}}$ with $2^{V}/[v]$ in the natural
way. Then $M/v$ is the binary matroid on $V\backslash\{v\}$ with
$Z(M/v)=(Z(M)+[v])/[v]$.

Our first indication of the importance of local complementation for adjacency
matroids is the fact that the matroid minors $M_{A}(G)/v$ and $M_{A}(G)-v$ can
always be obtained by deleting $v$ from graphs related to $G$ through local complementation.

\begin{theorem}
\label{loopcontr}If $v\in V(G)$ is a looped vertex then $M_{A}(G)/v$ $=$
$M_{A}(G^{v}-v)$.
\end{theorem}

\begin{theorem}
\label{nonloopcontr}Suppose $v$ is an unlooped vertex of $G$.

\begin{enumerate}
\item If $v$ is isolated then $M_{A}(G)/v$ $=$ $M_{A}(G-v)$.

\item If $w$ is an unlooped neighbor of $v$ then $M_{A}(G)/v$ $=$
$M_{A}((G^{w})^{v}-v)$.

\item If $w$ is a looped neighbor of $G$ then $M_{A}(G)/v$ $=$ $M_{A}%
(((G^{v})^{w})^{v}-v)$.
\end{enumerate}
\end{theorem}

\begin{theorem}
\label{deldel}If $v$ is not a coloop of $M_{A}(G)$ then $M_{A}(G)-v$ $=$
$M_{A}(G-v)$.
\end{theorem}

In general, Theorem \ref{deldel} fails for coloops. For example, let $v$ and
$w$ be the vertices of the simple path $P_{2}$ of length two. Then $w$ is
isolated and unlooped in $P_{2}-v$; consequently $\mathcal{C}(M_{A}(P_{2}-v))$
$=$ $\{\{w\}\}$ even though $\mathcal{C}(M_{A}(P_{2}))$ $=$ $\varnothing$.

Observe that if $M$ is a matroid on $V$ and $v\in V$, then $M-v$ $=$ $M/v$ if
and only if $v$ is either a loop or a coloop. (For if $v$ appears in any
circuit $\gamma$ with $\left\vert \gamma\right\vert >1$ then $\gamma
\backslash\{v\}$ contains a circuit of $M/v$ but $\gamma\backslash\{v\}$
contains no circuit of $M-v$.) It follows that the failure of Theorem
\ref{deldel} for coloops is not a significant inconvenience: if $v$ is a
coloop of $M_{A}(G)$ then we may refer to Theorem \ref{loopcontr} or Theorem
\ref{nonloopcontr} to describe $M_{A}(G)-v$ $=$ $M_{A}(G)/v$.

Another instance of the importance of local complementation is the fact that
matroid deletions from $M_{A}(G)$ and $M_{A}(G^{v})$ always coincide.

\begin{theorem}
\label{lcmat2}If $v\in V(G)$ then $M_{A}(G)-v$ $=$ $M_{A}(G^{v})-v$.
\end{theorem}

The connection between the entire matroids $M_{A}(G)$ and $M_{A}(G^{v})$ is
more complicated. Recall that if $M$ and $M^{\prime}$ are matroids on disjoint
sets $V$ and $V^{\prime}$ then their \emph{direct sum} $M\oplus M^{\prime}%
$\ is the matroid on $V\cup V^{\prime}$ with $\mathcal{C}(M\oplus M^{\prime
})=\mathcal{C}(M)\cup\mathcal{C}(M^{\prime})$. Also, if $v$ is a single
element then $U_{1,1}(\{v\})$ denotes the matroid on $\{v\}$ in which $v$ is a
coloop, and $U_{1,0}(\{v\})$ denotes the matroid on $\{v\}$ in which $v$ is a
loop. Clearly $M=(M-v)\oplus U_{1,1}(\{v\})$ if and only if $v$ is a coloop of
$M$, and $M=(M-v)\oplus U_{1,0}(\{v\})$ if and only if $v$ is a loop of $M$.

\begin{theorem}
\label{lcmat}1. If $v\in V(G)$ is unlooped then $M_{A}(G^{v})$ $=$ $M_{A}(G)$.

2. If $v\in V(G)$ is a coloop of both $M_{A}(G)$ and $M_{A}(G^{v})$, then
$M_{A}(G^{v})$ $=$ $M_{A}(G)$.

3. If $v\in V(G)$ is looped and not a coloop of one of $M_{A}(G),M_{A}(G^{v}%
)$, then $v$ is a coloop of the other and $M_{A}(G^{v})\ncong M_{A}(G)$. More
specifically, if $\{M_{A}(G),M_{A}(G^{v})\}=\{M_{1},M_{2}\}$ with $v$ not a
coloop of $M_{1}$, then $v$ is a coloop of $M_{2}$ and $M_{2}=(M_{1}-v)\oplus
U_{1,1}(\{v\})$.
\end{theorem}

Unlike Theorem \ref{iso}, Theorem \ref{lcmat} does not explain all
isomorphisms of adjacency matroids. For instance, the simple path of length
two has the same adjacency matroid as a disconnected graph consisting of two
looped vertices; but local complementation cannot disconnect a connected graph.

Theorems \ref{loopcontr} -- \ref{lcmat} are proven in\ Section 5, using
elementary linear algebra. In Section 6 we discuss Theorem \ref{trio}, which
provides another illustration of the connections tying adjacency matroids to
delta-matroids and the theory of the interlace polynomials. It is a matroid
version of Lemma 2 of Balister, Bollob\'{a}s, Cutler, and Pebody \cite{BBCS}.
A preliminary definition will be useful.

\begin{definition}
If $v$ is a vertex of $G$ then $G(v)$ denotes the graph obtained from $G$ by
removing every loop incident on $v$, $G(v,\ell)$ denotes the graph obtained
from $G(v)$ by attaching a loop at $v$, and $G(v,\ell i)$ denotes the graph
obtained from $G(v,\ell)$ by isolating $v$ (i.e., removing all non-loop edges
incident on $v$).
\end{definition}

\begin{theorem}
\label{trio}Let $v$ be a vertex of $G$. Then two of the three adjacency
matroids $M_{A}(G(v))$, $M_{A}(G(v,\ell))$, $M_{A}(G(v,\ell i))$ are the same,
and the other is different. The cycle space of the different matroid contains
the cycle space shared by the two that are the same, and its dimension is
greater by 1.
\end{theorem}

The adjacency matroid $M_{A}(G(v,\ell i))$ may be described in two other ways.
As $v$ is an isolated, looped vertex of $G(v,\ell i)$, it is a coloop of
$M_{A}(G(v,\ell i))$; hence $M_{A}(G(v,\ell i))=M_{A}(G(v,\ell i)-v)\oplus
U_{1,1}(\{v\})=M_{A}(G-v)\oplus U_{1,1}(\{v\})$. Another description comes
from Theorem \ref{loopcontr}, which tells us that $M_{A}(G-v)=M_{A}%
(G^{v}(v,\ell))/v$. Also, the fact that $v$ is not a loop of $M_{A}%
(G^{v}(v,\ell))$ implies that $M_{A}(G^{v}(v,\ell))/v$ and $M_{A}(G^{v}%
(v,\ell))$ have the same nullity. For ease of reference we state these
observations as a proposition.

\begin{proposition}
\label{new}If $v\in V(G)$ then $v$ is a coloop of
\[
M_{A}(G(v,\ell i))=M_{A}(G-v)\oplus U_{1,1}(\{v\})=(M_{A}(G^{v}(v,\ell
))/v)\oplus U_{1,1}(\{v\}).
\]
Consequently $M_{A}(G(v,\ell i))$, $M_{A}(G-v)$, $M_{A}(G^{v}(v,\ell))/v$ and
$M_{A}(G^{v}(v,\ell))$ all have the same nullity.
\end{proposition}

Recall that according to property (ii) above, $v$ must be a coloop of at least
one of the matroids $M_{A}(G(v))$, $M_{A}(G(v,\ell))$. Consequently $v$ must
fall under one (and only one) of these cases:

\begin{enumerate}
\item $v$ is a coloop of both $M_{A}(G(v))$ and $M_{A}(G(v,\ell))$

\item $v$ is a coloop of $M_{A}(G(v))$ and not $M_{A}(G(v,\ell))$

\item $v$ is a coloop of $M_{A}(G(v,\ell))$ and not $M_{A}(G(v))$
\end{enumerate}

As each vertex of $G$ must fall under precisely one of the cases 1 -- 3, we
obtain a partition of $V(G)$ into three subsets. We refer to this partition of
$V(G)$ as the \emph{principal vertex tripartition} of $G$. In Section 7 we
prove that the three subsets of the principal vertex tripartition correspond
precisely to the three alternatives of Theorem \ref{trio}: $v$ falls under
case 1 if and only if $M_{A}(G(v))=M_{A}(G(v,\ell))$; $v$ falls under case 2
if and only if $M_{A}(G(v))=M_{A}(G(v,\ell i))$; and $v$ falls under case 3 if
and only if $M_{A}(G(v,\ell))=M_{A}(G(v,\ell i))$. Moreover, $v$ falls under
case 1 in $G$ if and only if $v$ falls under case 2 in $G^{v}$, and vice
versa. Theorem \ref{tripart} below includes all of these results, and also
gives a few more details; in particular, the final assertion of case 1
corrects an error in \cite{Tv}.

\begin{theorem}
\label{tripart} Let $v$ be a vertex of $G$. Then the list $M_{A}(G(v))$,
$M_{A}(G(v,\ell))$, $M_{A}(G(v,\ell i))$, $M_{A}(G^{v}(v))$, $M_{A}%
(G^{v}(v,\ell))$, $M_{A}(G^{v}(v,\ell i))$ includes either two or three
distinct matroids. Only one of these distinct matroids does not include $v$ as
a coloop, and this matroid determines the others as follows.

1. If $v$ is a coloop of both $M_{A}(G(v))$ and $M_{A}(G(v,\ell))$ then $v$ is
not a coloop of $M_{A}(G^{v}(v,\ell))$,
\begin{align*}
M_{A}(G(v,\ell i))  &  =(M_{A}(G^{v}(v,\ell))/v)\oplus U_{1,1}(\{v\})\text{
and}\\
M_{A}(G(v))  &  =M_{A}(G(v,\ell))=M_{A}(G^{v}(v))=M_{A}(G^{v}(v,\ell i))\\
&  =(M_{A}(G^{v}(v,\ell))-v)\oplus U_{1,1}(\{v\}).
\end{align*}
Moreover, $M_{A}(G(v,\ell i))$ and $M_{A}(G^{v}(v,\ell))$ have the same
nullity, say $\nu+1$; the nullity of $M_{A}(G(v))$ is $\nu$ and
\[
Z(M_{A}(G(v,\ell i)))\cap Z(M_{A}(G^{v}(v,\ell)))=Z(M_{A}(G(v))).
\]
This case requires that $G^{v}-v$ have at least one looped vertex.

2. If $v$ is not a coloop of $M_{A}(G(v,\ell))$ then the assertions of case 1
hold, with the roles of $G$ and $G^{v}$ interchanged.

3. If $v$ is not a coloop of $M_{A}(G(v))$ then%
\begin{align*}
M_{A}(G^{v}(v))  &  =M_{A}(G(v))\text{ and}\\
M_{A}(G(v,\ell))  &  =M_{A}(G^{v}(v,\ell))=M_{A}(G(v,\ell i))=M_{A}%
(G^{v}(v,\ell i))\\
&  =(M_{A}(G(v))-v)\oplus U_{1,1}(\{v\}).
\end{align*}
Moreover, the nullity of $M_{A}(G(v))$ is 1 more than the nullity of
$M_{A}(G(v,\ell))$, and
\[
Z(M_{A}(G(v,\ell)))\subset Z(M_{A}(G(v))).
\]

\end{theorem}

The principal vertex tripartition is reminiscent of the principal edge
tripartition of Rosenstiehl and Read \cite{RR}, but there is a fundamental
difference between the two tripartitions. The principal edge tripartition of
$G$ is determined by the polygon matroid of $G$, but the principal vertex
tripartition of $G$ is not determined by the adjacency matroid of $G$:

\begin{theorem}
\label{ind}The adjacency matroid and the principal vertex tripartition are
independent, in the sense that two graphs may have isomorphic adjacency
matroids and distinct principal vertex tripartitions, or nonisomorphic
adjacency matroids and equivalent principal vertex tripartitions.
\end{theorem}

After verifying Theorem \ref{ind} in Section 8, in
Sections~\ref{sec:recall_ss_dm} -- \ref{sec:dmat_alt_proofs_mat} we turn our
attention\ to the close connection between adjacency matroids and $\Delta$-matroids.

\begin{definition}
\label{dmatroid} \cite{Bu1} A \emph{delta-matroid} ($\Delta$-matroid for
short) is an ordered pair $D=(V,\sigma)$ consisting of a finite set $V$ and a
nonempty family $\sigma\subseteq2^{V}$ that satisfies the \emph{symmetric
exchange axiom}: For all $X,Y\in\sigma$ and all $u\in X\Delta Y$,
$X\Delta\{u\}\in\sigma$ or there is a $v\neq u\in X\Delta Y$ such that
$X\Delta\{u,v\}\in\sigma$ (or both).
\end{definition}

We often write $X\in D$ rather than $X\in\sigma$. The name $\Delta
$\emph{-matroid} reflects the fact that if $M$ is a matroid, the family of
bases $\mathcal{B}(M)$ satisfies the symmetric exchange property; indeed
$\mathcal{B}(M)$ satisfies the stronger \emph{basis exchange axiom}: if
$X,Y\in\mathcal{B}(M)$ and $u\in X\backslash Y$ then there is some $v\in
Y\backslash X$ with $X\Delta\{u,v\}\in\mathcal{B}(M)$.

\begin{definition}
\label{gmatroid} \cite{Bu1} If $G$ is a graph then its associated $\Delta
$-matroid is $\mathcal{D}_{G}=(V(G),\sigma)$ with
\[
\sigma=\{S\subseteq V(G)\mid\mathcal{A}(G[S])\text{ is nonsingular}\}\text{.}%
\]

\end{definition}

$\mathcal{D}_{G}$ is determined by the adjacency matroids $M_{A}(G[S])$:
$S\in\mathcal{D}_{G}$ if and only if $M_{A}(G[S])$ is a free matroid (i.e.,
$\mathcal{C}(M_{A}(G[S]))=\varnothing$). There is more to the relationship
between $\mathcal{D}_{G}$ and the matroids $M_{A}(G[S])$ than this obvious
observation, though. Recall that if $M$ is a matroid on $V$ then an
\emph{independent set }of $M$ is a subset of $V$ that contains no circuit of
$M$, and a \emph{basis} of $M$ is a maximal independent set. For $S\subseteq
V(G)$ let $\mathcal{I}(M_{A}(G[S]))$ and $\mathcal{B}(M_{A}(G[S]))$\ denote
the families of independent sets and bases (respectively) of the adjacency
matroid $M_{A}(G[S])$.

\begin{theorem}
\label{matdmat}Let $G$ be a graph, and suppose $S\subseteq V(G)$.

1. $M_{A}(G[S])$ is the matroid on $S$ with
\[
\mathcal{B}(M_{A}(G[S]))=\{\text{maximal }B\subseteq S\mid B\in\mathcal{D}%
_{G}\}.
\]

2. $M_{A}(G[S])$ is the matroid on $S$ with
\[
\mathcal{I}(M_{A}(G[S]))=\{I\subseteq S\mid\text{there is some }%
X\in\mathcal{D}_{G}\text{ with }I\subseteq X\subseteq S\}.
\]

3. $\mathcal{D}_{G[S]}$ is the $\Delta$-matroid $(S,\sigma)$ with%
\[
\sigma=%
{\displaystyle\bigcup\limits_{T\subseteq S}}
\mathcal{B}(M_{A}(G[T])).
\]

\end{theorem}

Although Theorem \ref{matdmat} applies to arbitrary subsets $S\subseteq V(G)$,
the heart of the theorem is the result that part 1 holds for $S=V(G)$; as
noted by Brijder and Hoogeboom \cite{BH2}, this result is a special case of
the strong principal minor theorem of Kodiyalam, Lam and Swan \cite{K}.

In Sections \ref{sec:del_con_minmax} and \ref{sec:dmat_alt_proofs_mat} we
reprove Theorems~\ref{loopcontr} -- \ref{lcmat} within the contexts of set
systems and $\Delta$-matroids. In particular,\ Theorem~\ref{loopcontr} and
Theorem~\ref{deldel} are generalized to set systems and $\Delta$-matroids,
respectively. In a similar vein, Theorem 14 of \cite{BH2} shows that some
aspects of the principal vertex tripartition extend to matroids associated
with arbitrary $\Delta$-matroids.

In\ Section 12 we discuss the connection between the interlace polynomials
introduced by Arratia, Bollob\'{a}s and Sorkin \cite{A1, A2, A} and the Tutte
polynomials of the adjacency matroids of a graph and its full subgraphs. This
connection seems to be fundamentally different from the connection between the
one-variable interlace polynomial of a planar circle graph and the Tutte
polynomial of an associated checkerboard graph, discussed by Arratia,
Bollob\'{a}s and Sorkin \cite{A2} and Ellis-Monaghan and Sarmiento \cite{EMS}.

\section{Theorems \ref{bin} and \ref{Jae}}

For the convenience of the reader, in this section we provide proofs of
theorems of Ghouila-Houri \cite{GH} and Jaeger \cite{J1} mentioned in the introduction.

It is well known that axiom 4 of Definition \ref{binmat} may be replaced by
the following seemingly stronger requirement \cite{O, We, W2}:

4$^{\prime}$. If $C_{1},C_{2},...,C_{k}\in\mathcal{C}(M)$ do not sum to
$\varnothing$ in $2^{V}$ then there are pairwise disjoint $C_{1}^{\prime
},...,C_{k^{\prime}}^{\prime}\in\mathcal{C}(M)$ such that
\[%
{\displaystyle\sum\limits_{i=1}^{k}}
C_{i}=%
{\displaystyle\bigcup\limits_{i=1}^{k^{\prime}}}
C_{i}^{\prime}.
\]

This axiom is useful in the proof of Theorem \ref{bin}:

Let $M$ be a binary matroid on $V$. We can certainly construct $Z(M)$ from
$\mathcal{C}(M)$, using the addition of $2^{V}$. It turns out that we can also
construct $\mathcal{C}(M)$ from $Z(M)$:%
\[
\mathcal{C}(M)=\{\text{minimal nonempty subsets of }V\text{ that appear in
}Z(M)\}.
\]
The proof is simple: axiom 4$^{\prime}$ implies that every nonempty element of
$Z(M)$ contains a circuit, so every minimal nonempty element of $Z(M)$ is an
element of $\mathcal{C}(M)$; conversely, axiom 3 tells us that no circuit
contains another, so it is impossible for a circuit to contain a minimal
nonempty element of $Z(M)$ other than itself.

This implies that the function $Z$ is injective.

Now, let $W$ be any subspace of $2^{V}$. If $W$ $=$ $\{\varnothing\}$ then $W$
$=$ $Z(U)$, where $U$ is the free matroid on $V$ (i.e., $\mathcal{C}(U)$ $=$
$\varnothing$). Suppose $\dim W\geq1$, and let $\mathcal{C}(W)$ be the set of
minimal nonempty subsets of $V$ that appear in $W$. It is a simple matter to
verify that $\mathcal{C}(W)$ satisfies Definition \ref{binmat}; hence
$\mathcal{C}(W)$ is the circuit-set of a binary matroid $M(W)$. As
$\mathcal{C}(M(W))$ $=$ $\mathcal{C}(W)\subseteq W$, and $Z(M(W))$ is spanned
by $\mathcal{C}(M(W))$, $Z(M(W))$ is a subspace of $W$.

Could $Z(M(W))$ be a proper subspace of $W$? If so, then there is some $w\in
W$ that is not an element of $Z(M(W))$. By definition, $\mathcal{C}(W)$ must
include some element $\gamma$ that is a subset of $w$. Then $w+\gamma$ $=$
$w\Delta\gamma$ $=$ $w\backslash\gamma$ is also an element of $W-Z(M(W))$, and
its cardinality is strictly less than that of $w$. We deduce that $W-Z(M(W))$
does not have an element of smallest cardinality. This is ridiculous, so
$Z(M(W))$ cannot be a proper subspace of $W$.

It follows that the function $Z$ is surjective. $\blacksquare$

In light of Theorems \ref{bin} and \ref{iso} of Section 1, proving Theorem
\ref{Jae} is the same as proving the following.

\begin{proposition}
Let $A$ be a $k\times n$ matrix with entries in $GF(2)$. Then there is a
symmetric $n\times n$ matrix $B$ whose nullspace is the same as the right
nullspace of $A$.
\end{proposition}

\begin{proof}
If the right nullspace space of $A$ is $GF(2)^{n}$, the proposition is
satisfied by the zero matrix; if the right nullspace of $A$ is $\{\mathbf{0}%
\}$ then the proposition is satisfied by the identity matrix.

Otherwise, the right nullspace of $A$ is a proper subspace of $GF(2)^{n}$.
Using elementary row operations, we obtain from $A$ an $r\times n$ matrix $C$
in echelon form, which has the same right nullspace as $A$. (Here $r$ is the
rank of $A$.) There is a permutation $\pi$ of $\{1,...,n\}$ such that the
matrix obtained by permuting the columns of $C$ according to $\pi$ is of the
form%
\[
C^{\prime}=%
\begin{pmatrix}
I_{r} & C^{\prime\prime}%
\end{pmatrix}
\]
where $I_{r}$ is an identity matrix. If $(C^{\prime\prime})^{tr}$ denotes the
transpose of $C^{\prime\prime}$, then%
\[
B^{\prime}=%
\begin{pmatrix}
I_{r} & C^{\prime\prime}\\
(C^{\prime\prime})^{tr} & (C^{\prime\prime})^{tr}\cdot C^{\prime\prime}%
\end{pmatrix}
\]
is a symmetric matrix. $B^{\prime}$ has the same right nullspace as
$C^{\prime}$, for if $%
\begin{pmatrix}
I_{r} & C^{\prime\prime}%
\end{pmatrix}
\cdot\kappa$ $=$ $\mathbf{0}$ then certainly
\[%
\begin{pmatrix}
(C^{\prime\prime})^{tr} & (C^{\prime\prime})^{tr}\cdot C^{\prime\prime}%
\end{pmatrix}
\cdot\kappa=(C^{\prime\prime})^{tr}\cdot%
\begin{pmatrix}
I_{r} & C^{\prime\prime}%
\end{pmatrix}
\cdot\kappa=\mathbf{0.}%
\]
Consequently, a matrix $B$ satisfying the statement may be obtained by
permuting the rows and columns of $B^{\prime}$ according to $\pi^{-1}$.
\end{proof}

\section{Theorem \ref{corekerequiv}}

Our proof of Theorem \ref{Jae2} is rather different from Jaeger's original
argument \cite{J2}. We begin with the definition of \emph{touch-graphs}, which
appeared implicitly in Jaeger's later work \cite{J3} and were subsequently
discussed explicitly by Bouchet \cite{Bu4}. Recall that a \emph{trail} in a
graph is a walk which may include repeated vertices, but may not include
repeated edges. A closed trail is also called a \emph{circuit}; one such
circuit may contain another, so it is important to distinguish these circuits
from matroid circuits.

\begin{definition}
\label{touch} Let $F$ be a 4-regular graph. A \emph{circuit partition} or
\emph{Eulerian decomposition} of $F$ is a partition $P$ of the edge-set $E(F)$
into pairwise disjoint subsets, each of which is the edge-set of a closed
trail in $F$.
\end{definition}

\begin{definition}
If $P$ is a circuit partition of $F$ then the \emph{touch-graph} $Tch(P)$ is a
graph with a vertex for each element of $P$ and an edge for each vertex of
$F$; the edge corresponding to $v\in V(F)$ is incident on the vertex or
vertices corresponding to element(s) of $P$ incident at $v$.
\end{definition}

Observe that a walk $v_{1},e_{1},v_{2},...,e_{k-1},v_{k}$ in $F$ gives rise to
a walk in $Tch(P)$; for $1<i<k$ the edge of $Tch(P)$ corresponding to $v_{i}$
connects the vertex or vertices of $Tch(P)$ corresponding to the circuit(s) of
$P$ containing $e_{i-1}$ and $e_{i}$. Also, adjacent vertices of $Tch(P)$ must
correspond to circuits of $P$ contained in a single connected component of
$F$, because the edge of $Tch(P)$ connecting them corresponds to a vertex of
$F$ incident on both circuits. Consequently, there is a natural correspondence
between the connected components of $F$ and those of $Tch(P)$.

Recall that a connected 4-regular graph has at least one \emph{Euler circuit},
i.e., a closed trail that includes every edge. In general, a 4-regular graph
has at least one \emph{Euler system}, which includes one Euler circuit of each
connected component. Kotzig \cite{Kot} proved that if $P$ is a circuit
partition of a 4-regular graph $F$, then $F$ has an Euler system $C$ that
disagrees with $P$ at every vertex; that is, whenever one follows a circuit of
$C$ through a vertex, then one is not following a circuit of $P$. $C$ and $P$
are said to be $\emph{compatible}.$

An Euler system $C$ of a 4-regular graph $F$ has an associated
\emph{alternance graph }or \emph{interlacement graph} $\mathcal{I}(C)$,
defined as follows \cite{Bold, RR1}: The vertex-set of $\mathcal{I}(C)$ is
$V(F)$, and two vertices $v\neq w$ are adjacent in $\mathcal{I}(C)$ if and
only if there is some circuit of $C$ on which they appear in the order
$v...w...v...w$. We also use $\mathcal{I}(C)$ to denote the adjacency matrix
$\mathcal{A}(\mathcal{I}(C))$.

Suppose $P$ is a circuit partition of $F$, and $C$ is an Euler system.
Arbitrarily choose a preferred orientation of each circuit of $C$, and use
these orientations to direct the edges of $F$. Let $v\in V(F)$. Suppose we
choose a circuit of $P$ that is incident on $v$, and we walk toward $v$ on an
edge $e$ of this circuit that is in-directed according to the preferred
orientation. If we continue to follow this circuit of $P$, how do we leave
$v$? There are three possibilities: ($\phi$) we leave on the out-directed edge
we would use if we were following the incident circuit of $C$, ($\chi$) we
leave on the out-directed edge we would not use if we were following the
incident circuit of $C$, or ($\psi$) we leave on the in-directed edge we did
not use before.

Observe that if we were to choose the other in-directed edge instead of $e$,
or if we were to choose the other orientation of the incident circuit of $C$,
the $\phi$, $\chi$, $\psi$ designation would be the same. This designation is
called the \emph{transition} of $P$ at $v$; clearly $P$ is determined by its
transitions, and each of the $3^{\left\vert V(F)\right\vert }$ systems of
choices of transitions yields a circuit partition of $F$.

\begin{definition}
\label{mint} Under these circumstances, the \emph{relative interlacement graph
}$\mathcal{I}_{P}(C)$ is obtained from $\mathcal{I}(C)$ by removing each
vertex of type $\phi$, and attaching a loop to each vertex of type $\psi$.
\end{definition}

We also use $\mathcal{I}_{P}(C)$ to denote the adjacency matrix $\mathcal{A}%
(\mathcal{I}_{P}(C))$. An important property of $\mathcal{I}_{P}(C)$ is the
\emph{circuit-nullity formula}:

\begin{theorem}
\label{thmmint} If $F$ has $c(F)$ connected components then
\[
\nu(\mathcal{I}_{P}(C))=\left\vert P\right\vert -c(F)\text{,}%
\]
where $\nu$ denotes the $GF(2)$-nullity.
\end{theorem}

We refer to \cite{Tbn, T3, T5} for discussions and proofs of the
circuit-nullity formula. These references are recent, but special cases of
Theorem \ref{thmmint} have been known for almost 100 years; the earliest of
these special cases seems to be the one that appears in Brahana's study of
systems of curves on surfaces \cite{Br}. Jaeger \cite{J3} proved the special
case of the circuit-nullity formula in which $C$ and $P$ are compatible.

Observe that the rank of the polygon matroid of $Tch(P)$ is $\left\vert
V(Tch(P))\right\vert -c(Tch(P))$, which equals $\left\vert P\right\vert
-c(F)$. The circuit-nullity formula tells us that this rank equals
$\nu(\mathcal{I}_{P}(C))$. In case $C$ and $P$ are compatible, Jaeger
sharpened this equality by identifying the nullspace of $\mathcal{I}_{P}(C)$:

\begin{theorem}
\label{coreker} \cite{J3} Let $F$ be a 4-regular graph with a circuit
partition $P$ and a compatible Euler system $C.$ Then the nullspace of
$\mathcal{I}_{P}(C)$ and the cycle space of $Tch(P)$ are orthogonal
complements in $GF(2)^{V(F)}$.
\end{theorem}

Theorem \ref{coreker} suffices for our present purposes, but we might mention
that a modified version of the result holds in the general (non-compatible)
case; see \cite{T5}.

In view of Theorem \ref{bin}, Theorem \ref{coreker} is equivalent to the
following precise version of Theorem \ref{corekerequiv}:

\begin{theorem}
\label{corekerequi}Let $F$ be a 4-regular graph with a circuit partition $P$
and a compatible Euler system $C.$ Then $M_{A}(\mathcal{I}_{P}(C))^{\ast}$ is
the polygon matroid of $Tch(P)$.
\end{theorem}

We sketch a quick proof of Theorem \ref{corekerequi}, which uses the
circuit-nullity formula.

\begin{center}%
\begin{figure}
[ptb]
\begin{center}
\includegraphics[
height=2.687in,
width=4.2765in
]%
{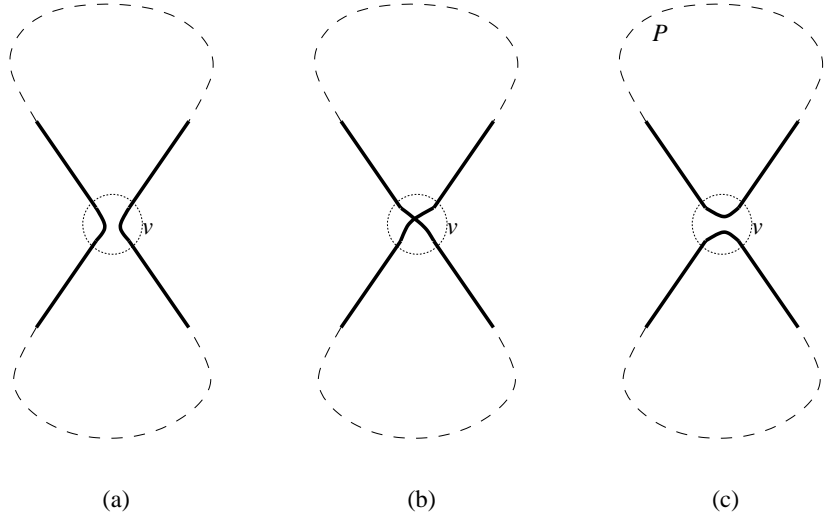}%
\caption{Given choices of transitions at other vertices, two of the three
transitions at $v$ produce touch-graphs in which $v$ is a loop.}%
\label{merge}%
\end{center}
\end{figure}

\end{center}

\begin{lemma}
\label{ranklem}Suppose $F$ is a 4-regular graph with a circuit partition $P$
and an Euler system $C$, and suppose $x$ is a vertex of $\mathcal{I}_{P}(C)$.
Two different circuits of $P$ are incident at $x$ if and only if
$\rho(\mathcal{I}_{P}(C))=\rho(\mathcal{I}_{P}(C)-x)$.
\end{lemma}

\begin{proof}
Let $P^{\prime}$ be the circuit partition obtained from $P$ by changing the
\emph{transition} at $x$: the circuit of $P^{\prime}$ incident at $x$ follows
the circuit of $C$ incident at $x$.

If two different circuits of $P$ are incident at $x$, then $P^{\prime}$ is
obtained from $P$ by uniting these two circuits. (See Figure \ref{merge}.) It
follows from the circuit-nullity formula that $\nu(\mathcal{I}_{P^{\prime}%
}(C))=\nu(\mathcal{I}_{P}(C))-1$. As $\mathcal{I}_{P^{\prime}}(C)$ is
$\mathcal{I}_{P}(C)-x$, we conclude that $\rho(\mathcal{I}_{P}(C))=\rho
(\mathcal{I}_{P}(C)-x)$. Conversely, if $\rho(\mathcal{I}_{P}(C))=\rho
(\mathcal{I}_{P}(C)-x)$ then $\nu(\mathcal{I}_{P}(C))=1+\nu(\mathcal{I}%
_{P^{\prime}}(C))$, so the circuit-nullity formula tells us that $\left\vert
P\right\vert =1+\left\vert P^{\prime}\right\vert $; consequently two different
circuits of $P$ are incident at $x$.
\end{proof}

Recall that an \emph{independent set} of a matroid is a set that contains no circuit.

\begin{corollary}
\label{rankcor}Let $F$ be a 4-regular graph with a circuit partition $P$ and
an Euler system $C$.\ Suppose $k\geq1$ and $X=\{x_{1},...,x_{k}\}\subseteq
V(\mathcal{I}_{P}(C))$. For $1\leq i\leq k$ let $P_{i}$ be the circuit
partition that involves the same transition as $C$ at each of $x_{1}%
,...,x_{i}$, and the same transition as $P$ at every other vertex. Then the
following are equivalent:

\begin{enumerate}
\item $X$ is an independent set of the polygon matroid of $Tch(P).$

\item $\rho(\mathcal{I}_{P}(C)-X)=\rho(\mathcal{I}_{P}(C))$.

\item $\left\vert P_{i}\right\vert =\left\vert P\right\vert -i$ for each
$i\geq1$.
\end{enumerate}
\end{corollary}

\begin{proof}
If $X=\varnothing$ then all three properties hold, and if $\left\vert
X\right\vert =1$ then their equivalence follows from Lemma \ref{ranklem} and
the fact that $x_{1}$ is a non-loop edge of $Tch(P)$ if and only if two
different circuits of $P$ are incident there.

We proceed by induction on $\left\vert X\right\vert =k>1$. Observe that
$\mathcal{I}_{P_{i}}(C)=\mathcal{I}_{P}(C)-\{x_{1},...,x_{i}\}$ for each $i>0$.

If any one of the three properties fails for $\{x_{1},...,x_{k-1}\}$ then by
induction, all three fail for $\{x_{1},...,x_{k-1}\}$; clearly then all three
also fail for $X$.

Suppose all three properties hold for $\{x_{1},...,x_{k-1}\}$. Property 3
implies that if $1\leq i<k$, then $Tch(P_{i})$ is obtained from $Tch(P)$ by
first contracting the edges corresponding to $x_{1},...,x_{i}$, and then
attaching loops to the vertices of $Tch(P)/\{x_{1},...,x_{i}\}$ corresponding
to $x_{1},...,x_{i}$.

If property 1 holds for $X$ then $x_{k}$ is a non-loop edge in $Tch(P)/\{x_{1}%
,...,x_{k-1}\}$, and hence also in $Tch(P_{k-1})$. That is, two different
circuits of $P_{k-1}$ are incident at $x_{k}$. Lemma \ref{ranklem} then
implies that $\rho(\mathcal{I}_{P_{k-1}}(C)-x_{k})=\rho(\mathcal{I}_{P_{k-1}%
}(C))$; this equals $\rho(\mathcal{I}_{P}(C))$ because property 2 holds for
$\{x_{1},...,x_{k-1}\}$. As $\mathcal{I}_{P_{k-1}}(C)-x_{k}=\mathcal{I}%
_{P}(C)-X$, property 2 holds for $X$.

If property 2 holds for $X$ then as it also holds for $\{x_{1},...,x_{k-1}\}$,
$\rho(\mathcal{I}_{P_{k-1}}(C)-x_{k})=\rho(\mathcal{I}_{P_{k-1}}(C))$. Lemma
\ref{ranklem} tells us that two different circuits of $P_{k-1}$ are incident
at $x_{k}$, so $\left\vert P_{k}\right\vert =\left\vert P_{k-1}\right\vert
-1$. As property 3 holds for $\{x_{1},...,x_{k-1}\}$, it also holds for $X$.

Finally, if property 3 holds for $X$ then $\left\vert P_{k}\right\vert
=\left\vert P_{k-1}\right\vert -1$, so two different circuits of $P_{k-1}$ are
incident at $x_{k}$. That is, the edge of $Tch(P_{k-1})$ corresponding to
$x_{k}$ is not a loop. It follows that the edge of $Tch(P)$ corresponding to
$x_{k}$ is not a loop in $Tch(P)/\{x_{1},...,x_{k-1}\}$. As property 1 holds
for $\{x_{1},...,x_{k-1}\}$, it follows that $X$ is independent in $Tch(P)$.
\end{proof}

The last step of the proof of Theorem \ref{corekerequi} is the following.

\begin{proposition}
Let $F$ be a 4-regular graph with a circuit partition $P$ and an Euler system
$C$.\ Suppose $X=\{x_{1},...,x_{k}\}\subseteq V(\mathcal{I}_{P}(C))$. Then
$\rho(\mathcal{I}_{P}(C)-X)=\rho(\mathcal{I}_{P}(C))$ if and only if $X$ is an
independent set of $M_{A}(\mathcal{I}_{P}(C))^{\ast}$.
\end{proposition}

\begin{proof}
Recall that $M_{A}(\mathcal{I}_{P}(C))^{\ast}$ is defined by the fact that its
independent sets are the complements of the spanning sets of $M_{A}%
(\mathcal{I}_{P}(C))$. That is, $X$ is an independent set of $M_{A}%
(\mathcal{I}_{P}(C))^{\ast}$ if and only if the columns of $\mathcal{I}%
_{P}(C)$ corresponding to elements of $Y=V(\mathcal{I}_{P}(C))-X$ span the
column space $W$ of $\mathcal{I}_{P}(C)$.

If $\rho(\mathcal{I}_{P}(C)-X)=\rho(\mathcal{I}_{P}(C))$ then $Y$ has
$\rho(\mathcal{I}_{P}(C))$ independent columns, and these must certainly span
$W$.

Conversely, if $W$ is spanned by the columns of $\mathcal{I}_{P}(C)$
corresponding to elements of $Y$ then there must be a subset $B\subseteq Y$
such that the columns of $\mathcal{I}_{P}(C)$ corresponding to elements of $B$
constitute a basis of $W$. As $\mathcal{I}_{P}(C)$ is symmetric, the strong
principal minor theorem of Kodiyalam, Lam and Swan \cite{K} (see Theorem
\ref{spmt} below) tells us that the principal submatrix of $\mathcal{I}%
_{P}(C)$ corresponding to $B$ is nonsingular. This principal submatrix is a
submatrix of $\mathcal{I}_{P}(C)-X$, so $\rho(\mathcal{I}_{P}(C)-X)\geq
\left\vert B\right\vert =\rho(\mathcal{I}_{P}(C))$. The opposite inequality is
obvious, so $\rho(\mathcal{I}_{P}(C)-X)=\rho(\mathcal{I}_{P}(C))$.
\end{proof}

As mentioned in the introduction, the duality between touch-graphs and the
adjacency matroids of looped circle graphs motivates many properties of
general adjacency matroids. For instance, suppose $C$ is an Euler system of a
4-regular graph $F$, $P$ is a circuit partition compatible with $C$, and $v\in
V(F)$. Then Figure \ref{merge} makes it clear that if $P^{\prime}$ is the
circuit partition obtained from $P$ by changing the transition at $v$ to the
other transition that does not appear in $C$, then $v$ is a loop in at least
one of the graphs $Tch(P)$, $Tch(P^{\prime})$. It follows that $v$ is a loop
of at least one of the matroids $M_{A}(\mathcal{I}_{P}(C))^{\ast}$,
$M_{A}(\mathcal{I}_{P^{\prime}}(C))^{\ast}$. Thus touch-graphs motivate
property (ii) of the introduction.

Similarly, touch-graphs motivate Theorem \ref{lcmat} of the introduction. Let
$F$ be a 4-regular graph with a vertex $v$ and an Euler system $C$. Kotzig
\cite{Kot} observed that $F$ also has an\ Euler system $C\ast v$, which
involves the same transition as $C$ at every vertex other than $v$, and at $v$
involves the transition that is orientation-inconsistent with respect to $C$.
Let $P$ be a circuit partition of $F$ that is compatible with $C$. Then the
following observations explain the duals of the three assertions of Theorem
\ref{lcmat}. 1. If $v$ is unlooped in $\mathcal{I}_{P}(C)$ then $P$ involves
the transition at $v$ that is consistent with the orientation of $C$, $P$ is
compatible with $C\ast v$ and $\mathcal{I}_{P}(C\ast v)$ is the local
complement $\mathcal{I}_{P}(C)^{v}$. Then $M_{A}(\mathcal{I}_{P}%
(C))=M_{A}(\mathcal{I}_{P}(C)^{v})$ because both are dual to the polygon
matroid of $Tch(P)$. 2. Suppose $v$ is looped in $\mathcal{I}_{P}(C)$, and let
$P^{\prime}$ be the circuit partition obtained from $P$ by changing the
transition at $v$ to the one that is orientation-inconsistent with $C$. Then
$\mathcal{I}_{P}(C)^{v}=\mathcal{I}_{P^{\prime}}(C\ast v)$. If $v$ is a loop
in both $Tch(P)$ and $Tch(P^{\prime})$, then $Tch(P)=Tch(P^{\prime})$ and
hence $M_{A}(\mathcal{I}_{P}(C))=M_{A}(\mathcal{I}_{P}(C)^{v})$ because both
are dual to the polygon matroid of $Tch(P)=Tch(P^{\prime})$. 3. Suppose $v$ is
looped in $\mathcal{I}_{P}(C)$, but not a loop in $Tch(P)$. Again, let
$P^{\prime}$ be the circuit partition obtained from $P$ by changing the
transition at $v$ to the one that is orientation-inconsistent with $C$. Then
$Tch(P^{\prime})$ is the graph obtained from $Tch(P)$ by contracting $v$ and
then attaching a loop at the vertex corresponding to $v$, so the dual of the
polygon matroid of $Tch(P^{\prime})$ is isomorphic to the matroid obtained
from the dual of the polygon matroid of $Tch(P)$ by deleting $v$ and replacing
it with a coloop.

\section{Theorem \ref{ubiq}}

Let $G$ be any graph with no isolated, unlooped vertex. Then a 4-regular graph
$F$ with a distinguished circuit partition $P$ may be constructed from $G$ in
two steps, as follows.%
\begin{figure}
[ptb]
\begin{center}
\includegraphics[
trim=0.800687in 0.933839in 1.074454in 0.937048in,
height=6.6452in,
width=4.8066in
]%
{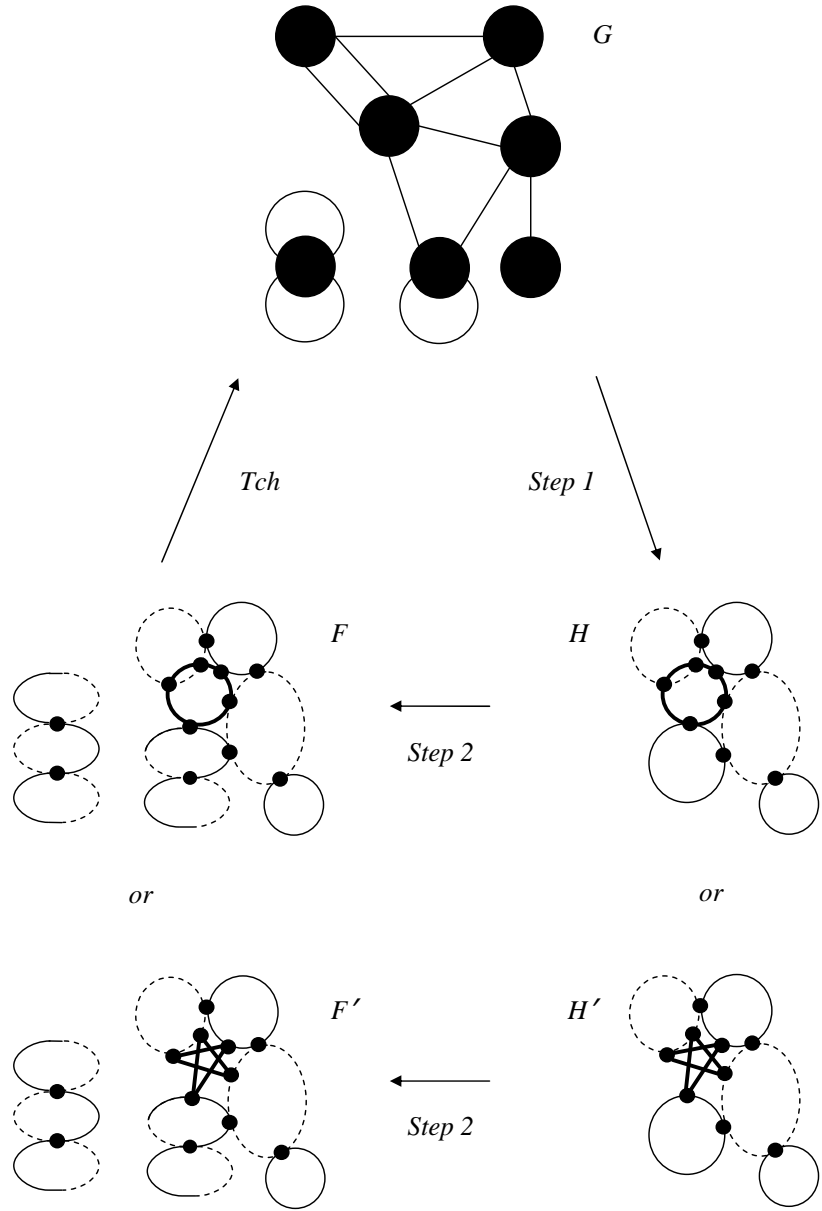}%
\caption{Bold, dashed and plain \textquotedblleft edge
styles\textquotedblright\ indicate distinguished circuit partitions $P$ and
$P^{\prime}$ with $Tch(P)$ and $Tch(P^{\prime})$ isomorphic to $G$.}%
\label{circgraf3b}%
\end{center}
\end{figure}

Step 1. For each vertex $v\in V(G)$, list the incident non-loop edges as
$e_{1}^{v},...,e_{d_{v}}^{v}$ in some order. Then construct a 4-regular graph
$H$ with a vertex for each non-loop edge of $G$, in such a\ way that for each
$v\in V(G)$, $H$ has $d_{v}$ edges; the $i^{th}$ of these edges connects the
vertex corresponding to $e_{i}^{v}$ to the vertex corresponding to
$e_{i+1}^{v}$ (with subindices considered modulo $d_{v}$). $H$ has a
distinguished circuit partition whose elements correspond to the non-isolated
vertices of $G$.

Step 2. Suppose $G$ has $\ell$ loops. If $\ell=0$ then let $H=F$, and let $P$
be the distinguished circuit partition. Otherwise, let $H=H_{0}$ and list the
loops of $G$ as $e_{1},...,e_{\ell}$ in some order. Suppose $1\leq i\leq\ell$
and $H_{i-1}$ has been constructed with a distinguished circuit partition
whose circuits correspond to some of the vertices of $G$. If the vertex of $G$
incident on $e_{i}$ does not correspond to a distinguished circuit of
$H_{i-1}$, then $H_{i}$ is obtained from $H_{i-1}$ by adjoining a new vertex
with a \textquotedblleft figure eight\textquotedblright\ on it; that is, a
single distinguished circuit consisting of two loops. The new distinguished
circuit corresponds to the vertex of $G$ incident on $e_{i}$, and the new
vertex corresponds to $e_{i}$. If the vertex of $G$ incident on $e_{i}$ does
correspond to a distinguished circuit of $H_{i-1}$, then $H_{i}$ is obtained
from $H_{i-1}$ by introducing a looped vertex \textquotedblleft in the
middle\textquotedblright\ of some edge of this circuit; the corresponding
distinguished circuit of $H_{i}$ includes the three new edges and all the
other edges of the original distinguished circuit of $H_{i-1}$.

At the end of this process we have obtained $H_{\ell}=F$, with a distinguished
circuit partition $P$ such that $Tch(P)\cong G$.

The construction is illustrated in Figure \ref{circgraf3b}, where
distinguished circuit partitions are indicated by the convention that when
following a distinguished circuit through a vertex, one does not change the
\textquotedblleft edge style\textquotedblright\ (bold, dashed or plain). The
\textquotedblleft edge style\textquotedblright\ may change in the middle of an
edge, though.

Note that there is considerable freedom in the construction, both in choosing
the edge-orders $e_{1}^{v},...,e_{d_{v}}^{v}$ and in locating the
non-isolated, looped vertices introduced in Step 2. Consequently the resulting
4-regular graph is far from unique; in Figure \ref{circgraf3b}, $F$ is a
planar graph with two pairs of parallel edges, and $F^{\prime}$ is a
non-planar graph with only one pair of parallel edges.

\section{Local complementation and matroid minors}

The reader familiar with the interlace polynomials of Arratia, Bollob\'{a}s
and Sorkin \cite{A1, A2, A} will recognize some of the concepts and notation
that appear in our discussion of adjacency matroids, but it is important to
keep a significant difference in mind: The interlace polynomials of $G$ are
related to principal submatrices of $\mathcal{A}(G)$, i.e., square submatrices
obtained from $\mathcal{A}(G)$ by removing some columns and the corresponding
rows. The adjacency matroid of $G$, instead, is related to rectangular
submatrices obtained by removing only columns from $\mathcal{A}(G)$.

\subsection{Theorems \ref{loopcontr} and \ref{lcmat2}}

Suppose $v\in V(G)$; let
\[
\mathcal{A}(G)=%
\begin{pmatrix}
\ast & \mathbf{1} & \mathbf{0}\\
\mathbf{1} & A & B\\
\mathbf{0} & C & D
\end{pmatrix}
\text{ and }\mathcal{A}(G^{v})=%
\begin{pmatrix}
\ast & \mathbf{1} & \mathbf{0}\\
\mathbf{1} & A^{c} & B\\
\mathbf{0} & C & D
\end{pmatrix}
.
\]
Here bold numerals indicate rows and columns with all entries equal, the first
row and column correspond to $v$, and $A^{c}$ is the matrix obtained by
toggling all the entries of $A$. To prove Theorem \ref{lcmat2}, observe that
elementary row operations transform
\[%
\begin{pmatrix}
\mathbf{1} & \mathbf{0}\\
A & B\\
C & D
\end{pmatrix}
\text{ into }%
\begin{pmatrix}
\mathbf{1} & \mathbf{0}\\
A^{c} & B\\
C & D
\end{pmatrix}
\text{ .}%
\]
It follows that if $\kappa$ is a column vector then%
\[%
\begin{pmatrix}
\mathbf{1} & \mathbf{0}\\
A & B\\
C & D
\end{pmatrix}
\cdot\kappa=\mathbf{0}\text{ if and only if }%
\begin{pmatrix}
\mathbf{1} & \mathbf{0}\\
A^{c} & B\\
C & D
\end{pmatrix}
\cdot\kappa=\mathbf{0}\text{.}%
\]
That is, these two matrices have the same right nullspace. It follows that
$Z(M_{A}(G)-v)$ $=$ $Z(M_{A}(G^{v})-v)$, and hence $M_{A}(G)-v$ $=$
$M_{A}(G^{v})-v$, as asserted by Theorem \ref{lcmat2}.

Theorem \ref{loopcontr} follows from another calculation using elementary row
operations. Matrices of the forms%
\[%
\begin{pmatrix}
1 & \mathbf{1} & \mathbf{0}\\
\mathbf{1} & A & B\\
\mathbf{0} & C & D
\end{pmatrix}
\text{, }%
\begin{pmatrix}
1 & \mathbf{1} & \mathbf{0}\\
\mathbf{0} & A^{c} & B\\
\mathbf{0} & C & D
\end{pmatrix}
\text{ and }%
\begin{pmatrix}
A^{c} & B\\
C & D
\end{pmatrix}
\text{ }%
\]
have the same $GF(2)$-nullity, so if $v$ is looped and $v\not \in S\subseteq
V(G)$ then $S\cup\{v\}$ contains a circuit of $M_{A}(G)$ if and only if $S$
contains a circuit of $M_{A}(G^{v}-v)$. It follows that if $v$ is looped, then
$M_{A}(G)/v$ $=$ $M_{A}(G^{v}-v)$.

\subsection{Theorem \ref{deldel} and the strong principal minor theorem}

We turn now to Theorem \ref{deldel}. Suppose $v$ is not a coloop of $M_{A}%
(G)$. Then $\mathcal{A}(G)$ is a symmetric matrix of the form%
\[%
\begin{pmatrix}
\ast & \mathbf{1} & \mathbf{0}\\
\mathbf{1} & A & B\\
\mathbf{0} & C & D
\end{pmatrix}
\text{,}%
\]
with the first column corresponding to $v$ and equal to the sum of certain
other columns. It follows that
\[%
\begin{pmatrix}
\mathbf{1} & \mathbf{0}\\
A & B\\
C & D
\end{pmatrix}
\text{ and }%
\begin{pmatrix}
A & B\\
C & D
\end{pmatrix}
\]
are related through elementary row operations, so these two matrices have the
same right nullspace: $Z(M_{A}(G)-v)$ $=$ $Z(M_{A}(G-v))$. Consequently
$M_{A}(G)-v$ $=$ $M_{A}(G-v)$, as asserted by Theorem \ref{deldel}.

By the way, note that this argument still applies if $v$ is a coloop, provided
that $v$ is not a coloop of the adjacency matroid of the graph obtained from
$G$ by toggling the loop status of $v$.

Theorem \ref{deldel} is equivalent to the following special case of the strong
principal minor theorem of Kodiyalam, Lam and Swan \cite{K}.

\begin{theorem}
\label{spmt}Let $A$ be a symmetric $n\times n$ matrix with entries in $GF(2)$
and let $S$ be a subset of $\{1,...,n\}$, of size $r=rank(A)$. Then the
columns of $A$ corresponding to elements of $S$ are linearly independent if
and only if the principal submatrix of $A$ corresponding to $S$ is nonsingular.
\end{theorem}

\begin{proof}
If the principal submatrix of $A$ corresponding to $S$ is nonsingular then its
columns must be linearly independent. Obviously then the corresponding columns
of $A$, which are obtained from the columns of the principal submatrix by
inserting rows corresponding to elements of $\{1,...,n\}\backslash S$, must
also be linearly independent.

The interesting part of the theorem is the converse: if the columns of $A$
corresponding to elements of $S$ form an $n\times r$ matrix of rank $r$, then
the $r\times r$ submatrix obtained by removing the rows corresponding to
elements of $\{1,...,n\}\backslash S$ is also of rank $r$. The proof is
simple: Let $A^{\prime}$ be the $n\times r$ submatrix of $A$ that includes
only the columns with indices from $S$; by hypothesis, $A$ and $A^{\prime}$
have the same column space. If $i\in\{1,...,n\}\backslash S$ then the $i$th
column of $A$ must be the sum of some columns with indices from $S$, and by
symmetry the $i$th row of $A$ must be the sum of some rows with indices from
$S$. Consequently the same is true of the $i$th row of $A^{\prime}$. It
follows that removing the $i$th row of $A^{\prime}$ for every $i\notin S$
yields an $r\times r$ submatrix with the same row space as $A^{\prime}$, and
hence the same rank as $A$.
\end{proof}

\subsection{Theorems \ref{nonloopcontr} and \ref{lcmat}}

To prove part 1 of Theorem \ref{lcmat}, suppose $v\in V(G)$ is unlooped. Let
\[
\mathcal{A}(G)=%
\begin{pmatrix}
0 & \mathbf{1} & \mathbf{0}\\
\mathbf{1} & A & B\\
\mathbf{0} & C & D
\end{pmatrix}
\text{ and }\mathcal{A}(G^{v})=%
\begin{pmatrix}
0 & \mathbf{1} & \mathbf{0}\\
\mathbf{1} & A^{c} & B\\
\mathbf{0} & C & D
\end{pmatrix}
.
\]
Elementary row operations transform $\mathcal{A}(G)$ into $\mathcal{A}(G^{c})
$, so these matrices have the same right nullspace. It follows that
$Z(M_{A}(G))$ $=$ $Z(M_{A}(G^{v}))$.

We are now ready to prove all three parts of Theorem \ref{nonloopcontr}. If
$v$ is unlooped and isolated then $v$ is a loop in $M_{A}(G)$, so $M_{A}(G)/v$
$=$ $M_{A}(G)-v$; Theorem \ref{deldel} tells us that $M_{A}(G)-v$ $=$
$M_{A}(G-v)$. If $v$ and $w$ are unlooped neighbors then part 1 of Theorem
\ref{lcmat} tells us that $M_{A}(G)$ $=$ $M_{A}(G^{w})$; $v$ is looped in
$G^{w}$, so Theorem \ref{loopcontr} tells us that $M_{A}(G^{w})/v$ $=$
$M_{A}((G^{w})^{v}-v)$. Finally, if $v$ is unlooped and $w$ is a looped
neighbor of $v$ in $G$ then part 1 of Theorem \ref{lcmat} tells us that
$M_{A}(G)$ $=$ $M_{A}(G^{v})$; $w$ is an unlooped neighbor of $v$ in $G^{v}$,
so the preceding sentence tells us that $M_{A}(G)/v$ $=$ $M_{A}(G^{v})/v$ $=$
$M_{A}(((G^{v})^{w})^{v}-v)$.

Turning to part 2 of Theorem \ref{lcmat}, suppose a looped vertex $v$ is a
coloop of both $M_{A}(G)$ and $M_{A}(G^{v})$. Then $\mathcal{C}(M_{A}(G))$ $=
$ $\mathcal{C}(M_{A}(G)-v)$ and $\mathcal{C}(M_{A}(G^{v}))$ $=$ $\mathcal{C}%
(M_{A}(G^{v})-v)$. Theorem \ref{lcmat2} tells us that $M_{A}(G)-v$ $=$
$M_{A}(G^{v})-v$, so $M_{A}(G)$ $=$ $M_{A}(G^{v})$.

Part 3 of Theorem \ref{lcmat} involves the following.

\begin{lemma}
\label{noncoloop}Suppose $v\in V(G)$. Then $v$ is not a coloop of $M_{A}(G)$
if and only if the three matrices
\[
\mathcal{A}(G)=%
\begin{pmatrix}
\ast & \mathbf{1} & \mathbf{0}\\
\mathbf{1} & A & B\\
\mathbf{0} & C & D
\end{pmatrix}
\text{, }%
\begin{pmatrix}
\mathbf{1} & \mathbf{0}\\
A & B\\
C & D
\end{pmatrix}
\text{ and }%
\begin{pmatrix}
\mathbf{0} & \mathbf{0}\\
A & B\\
C & D
\end{pmatrix}
\]
have the same rank over $GF(2)$. (Here the first row and column of
$\mathcal{A}(G)$ correspond to $v$.)
\end{lemma}

\begin{proof}
If the three matrices have the same rank then in particular, the first two
have the same rank. Consequently the first column of $\mathcal{A}(G)$ must
equal the sum of certain other columns; hence there is a circuit of $M_{A}(G)$
that contains $v$. Conversely, if $v$ is not a coloop of $M_{A}(G)$ then the
column of $\mathcal{A}(G)$ corresponding to $v$ must be the sum of the columns
corresponding to some subset $S_{v}\subseteq V(G)\backslash\{v\}$. By
symmetry, the sum of the rows corresponding to $S_{v}$ must equal the first
row of $\mathcal{A}(G)$.
\end{proof}

Suppose now that $v$ is a looped non-coloop of $M_{A}(G)$. The set $S_{v}$
must include an odd number of columns of $A$, to yield the diagonal entry
$\ast$ $=$ $1$ in the column of $\mathcal{A}(G)$ corresponding to $v$.
Replacing $A$ with $A^{c}$ toggles an odd number of summands in each row of
$A$, so the sum of the columns of
\[
\mathcal{A}(G^{v})=%
\begin{pmatrix}
1 & \mathbf{1} & \mathbf{0}\\
\mathbf{1} & A^{c} & B\\
\mathbf{0} & C & D
\end{pmatrix}
\]
corresponding to $S_{v}$ must be the column vector
\[%
\begin{pmatrix}
1\\
\mathbf{0}\\
\mathbf{0}%
\end{pmatrix}
.
\]
It follows that the $GF(2)$-ranks of
\[%
\begin{pmatrix}
\mathbf{1} & \mathbf{0}\\
A^{c} & B\\
C & D
\end{pmatrix}
\text{ and }%
\begin{pmatrix}
1 & \mathbf{0} & \mathbf{0}\\
\mathbf{0} & A^{c} & B\\
\mathbf{0} & C & D
\end{pmatrix}
\]
are the same. According to Lemma \ref{noncoloop}, $v$ cannot be a non-coloop
of $M_{A}(G^{v}).$

By the way, the same argument shows that removing the loop from $v$ cannot
produce a non-coloop in the adjacency matroid of the resulting graph. That is,
in the terminology of Section 4 $v$ is a triple coloop of $M_{A}(G^{v})$.

To complete the proof of part 3 of Theorem \ref{lcmat}, note that the fact
that $v$ is a coloop of $M_{A}(G^{v})$ implies that $\mathcal{C}(M_{A}%
(G^{v}))=\mathcal{C}(M_{A}(G^{v})-v)$. Theorem \ref{lcmat2} tells us that
$\mathcal{C}(M_{A}(G^{v})-v)=\mathcal{C}(M_{A}(G)-v)$, and the fact that $v$
is not a coloop of $M_{A}(G)$ implies that $\mathcal{C}(M_{A}(G)-v)$ is a
proper subset of $\mathcal{C}(M_{A}(G))$. It follows that $\left\vert
\mathcal{C}(M_{A}(G^{v}))\right\vert <\left\vert \mathcal{C}(M_{A}%
(G))\right\vert $, and consequently $M_{A}(G^{v})\ncong M_{A}(G)$. \ The
equality $M_{A}(G^{v})=(M_{A}(G^{v})-v)\oplus U_{1,1}(\{v\})$ follows
immediately from the fact that $v$ is a coloop of $M_{A}(G^{v})$, and this
equality implies $M_{A}(G^{v})=(M_{A}(G)-v)\oplus U_{1,1}(\{v\})$ by Theorem
\ref{lcmat2}.

\section{Theorem \ref{trio} and triple coloops}

Theorem \ref{trio} is essentially a result about the nullspaces of
$\mathcal{A}(G(v))$, $\mathcal{A}(G(v,\ell))$ and $\mathcal{A}(G(v,\ell i))$.
With a convenient order on the vertices of $G$, these three matrices are%
\[%
\begin{pmatrix}
0 & \mathbf{1} & \mathbf{0}\\
\mathbf{1} & A & B\\
\mathbf{0} & C & D
\end{pmatrix}
\text{, }%
\begin{pmatrix}
1 & \mathbf{1} & \mathbf{0}\\
\mathbf{1} & A & B\\
\mathbf{0} & C & D
\end{pmatrix}
\text{ and }%
\begin{pmatrix}
1 & \mathbf{0} & \mathbf{0}\\
\mathbf{0} & A & B\\
\mathbf{0} & C & D
\end{pmatrix}
\]
respectively. Theorem \ref{trio} asserts that two of the nullspaces are the
same, say of dimension $\nu$; the different nullspace contains them, and its
dimension is $\nu+1$. A proof of this statement is given in \cite{T3}.

Observe that no element of the nullspace of the right-hand matrix could
possibly have a nonzero first coordinate. Consequently $v$ does not appear in
any circuit of $M_{A}(G(v,\ell i))$; that is, $v$ is a coloop of
$M_{A}(G(v,\ell i))$, as noted in Proposition \ref{new} of the introduction.
In the special case that $M_{A}(G(v,\ell i))$ has a larger cycle space than
$M_{A}(G(v))=M_{A}(G(v,\ell))$, it follows that $v$ must also be a coloop of
$M_{A}(G(v))$ and $M_{A}(G(v,\ell))$. On the other hand, if $M_{A}(G(v))$ or
$M_{A}(G(v,\ell))$ has a larger cycle space than $M_{A}(G(v,\ell i))$ then
either the matrix displayed on the left or the matrix displayed in the center
has a larger nullspace than the one on the right. Clearly any vector in either
of these two nullspaces that is not in the nullspace of the right-hand matrix
must have a nonzero first coordinate; consequently $v$ is not a coloop of the
corresponding matroid. We deduce the following sharpened form of the result
(ii) mentioned in the introduction.

\begin{corollary}
\label{triple}If $v\in V(G)$ then $v$ is a coloop of $M_{A}(G(v,\ell i))$ and
at least one of the adjacency matroids $M_{A}(G(v))$, $M_{A}(G(v,\ell))$. It
is a coloop of all three if and only if
\[
Z(M_{A}(G(v)))=Z(M_{A}(G(v,\ell))\subset Z(M_{A}(G(v,\ell i))).
\]

\end{corollary}

The special case in which $v$ is a coloop of all three matroids is important
enough to merit a special name.

\begin{definition}
A vertex $v\in V(G)$ is a \emph{triple coloop} of $M_{A}(G)$ if it is a coloop
of $M_{A}(G(v))$, $M_{A}(G(v,\ell))$, and $M_{A}(G(v,\ell i))$.
\end{definition}

Note that the nomenclature is imprecise; although a triple coloop of
$M_{A}(G)$ is certainly a coloop of $M_{A}(G)$, it is not the matroid
structure of $M_{A}(G)$ that determines whether or not a vertex is a triple
coloop. We prefer this imprecise nomenclature over the alternative
\textquotedblleft$v$ is a triple coloop of $G$\textquotedblright\ because that
would also be confusing; a coloop (isthmus) of $G$ is an edge, not a vertex.

Using this notion, Theorems \ref{deldel} and \ref{lcmat} may be sharpened as follows.

\begin{theorem}
\label{lcmatsharp}1. If $v\in V(G)$ is unlooped then $M_{A}(G^{v})$ $=$
$M_{A}(G)$.

2. If a looped vertex $v\in V(G)$ is a coloop of both $M_{A}(G)$ and
$M_{A}(G^{v})$, then $M_{A}(G^{v})$ $=$ $M_{A}(G)$ and $v$ is not a triple
coloop of $M_{A}(G^{v})$ or $M_{A}(G)$.

3. If $v\in V(G)$ is looped and not a coloop of one of $M_{A}(G),M_{A}(G^{v}%
)$, then $v$ is a triple coloop of the other and $M_{A}(G^{v})\ncong M_{A}%
(G)$. More specifically, if $\{M_{A}(G),M_{A}(G^{v})\}=\{M_{1},M_{2}\}$ with
$v$ not a coloop of $M_{1}$, then $v$ is a triple coloop of $M_{2}$ and
$M_{2}=(M_{1}-v)\oplus U_{1,1}(\{v\})$.
\end{theorem}

\begin{theorem}
\label{deldelsharp}If $v\in V(G)$ is not a triple coloop of $M_{A}(G)$, then
$M_{A}(G)-v$ $=$ $M_{A}(G-v)$.
\end{theorem}

\begin{proof}
Part 1 of Theorem \ref{lcmatsharp} is the same as part 1 of Theorem
\ref{lcmat}. The proofs of Theorem \ref{deldelsharp} and part 3 of Theorem
\ref{lcmatsharp} are indicated in the preceding section; both are introduced
with the phrase \textquotedblleft by the way.\textquotedblright\ 

It remains to consider part 2 of Theorem \ref{lcmatsharp}. Suppose a looped
vertex $v\in V(G)$ is a triple coloop of $G$, and let%
\[
\mathcal{A}(G)=\mathcal{A}(G(v,\ell))=%
\begin{pmatrix}
1 & \mathbf{1} & \mathbf{0}\\
\mathbf{1} & A & B\\
\mathbf{0} & C & D
\end{pmatrix}
.
\]
According to Theorem \ref{trio} and Corollary \ref{triple},
\[
\nu(\mathcal{A}(G(v)))=\nu(\mathcal{A}(G(v,\ell)))=\nu(\mathcal{A}(G(v,\ell
i)))-1,
\]
where $\nu$ denotes the $GF(2)$-nullity. Observe that elementary row and
column operations transform $\mathcal{A}(G(v,\ell))$ into%
\[%
\begin{pmatrix}
1 & \mathbf{0} & \mathbf{0}\\
\mathbf{0} & A^{c} & B\\
\mathbf{0} & C & D
\end{pmatrix}
,
\]
which is $\mathcal{A}((G^{v}(v,\ell i))$. It follows that $\nu(\mathcal{A}%
(G(v,\ell)))$ $=$ $\nu(\mathcal{A}(G^{v}(v,\ell i)))$. Also, elementary row
operations transform
\[
\mathcal{A}(G(v))=%
\begin{pmatrix}
0 & \mathbf{1} & \mathbf{0}\\
\mathbf{1} & A & B\\
\mathbf{0} & C & D
\end{pmatrix}
\text{ into }%
\begin{pmatrix}
0 & \mathbf{1} & \mathbf{0}\\
\mathbf{1} & A^{c} & B\\
\mathbf{0} & C & D
\end{pmatrix}
=\mathcal{A}(G^{v}(v)),
\]
so $\nu(\mathcal{A}(G(v)))$ $=$ $\nu(\mathcal{A}(G^{v}(v)))$. We conclude that
$\nu(\mathcal{A}(G^{v}(v,\ell i)))$ $=$ $\nu(\mathcal{A}(G^{v}(v)))$;
according to Theorem \ref{trio} and Corollary \ref{triple}, this equality
implies that $v$ is not a coloop of $M_{A}(G^{v}(v,\ell))$ $=$ $M_{A}(G^{v})$.
\end{proof}

\section{The principal vertex tripartition}

Combining various results above, we see that if $v\in V(G)$ then the
relationships among the six adjacency matroids $M_{A}(G(v))$, $M_{A}%
(G(v,\ell))$, $M_{A}(G(v,\ell i))$, $M_{A}(G^{v}(v))$, $M_{A}(G^{v}(v,\ell))$,
$M_{A}(G^{v}(v,\ell i))$ must fall into one of three cases.

\textbf{Case 1}. Suppose $v$ is not a coloop of $M_{A}(G^{v}(v,\ell))$.
According to part 3 of Theorem \ref{lcmatsharp}, $v$ is a triple coloop of
$M_{A}(G(v,\ell))$. Theorem \ref{trio} tells us that%
\[
Z(M_{A}(G^{v}(v)))=Z(M_{A}(G^{v}(v,\ell i)))\subset Z(M_{A}(G^{v}(v,\ell)))
\]
and%
\[
Z(M_{A}(G(v)))=Z(M_{A}(G(v,\ell)))\subset Z(M_{A}(G(v,\ell i))).
\]
Part 1 of Theorem \ref{lcmat} tells us that $M_{A}(G^{v}(v))$ $=$
$M_{A}(G(v))$. $M_{A}(G(v,\ell i))$ has $v$ as a coloop, so the fact that $v$
is not a coloop of $M_{A}(G^{v}(v,\ell))$ implies that $M_{A}(G(v,\ell i))$
$\neq$ $M_{A}(G^{v}(v,\ell))$. All in all, we have the following:
$Z(M_{A}(G(v,\ell i)))$ and $Z(M_{A}(G^{v}(v,\ell)))$ are distinct nontrivial
subspaces of $2^{V(G)}$ with the same dimension, say $\nu+1$; their
intersection is of dimension $\nu$, and
\begin{align*}
Z(M_{A}(G^{v}(v,\ell i)))  &  =Z(M_{A}(G^{v}(v)))=Z(M_{A}(G(v)))=Z(M_{A}%
(G(v,\ell))\\
&  =Z(M_{A}(G(v,\ell i)))\cap Z(M_{A}(G^{v}(v,\ell))).
\end{align*}
The equality $M_{A}(G^{v}(v,\ell i))=(M_{A}(G^{v}(v,\ell))-v)\oplus
U_{1,1}(\{v\})$ follows from Theorem \ref{deldel} and Proposition \ref{new},
and $M_{A}(G(v,\ell i))=(M_{A}(G^{v}(v,\ell))/v)\oplus U_{1,1}(\{v\})$ follows
from Proposition \ref{new}.

\textbf{Case 2}. Suppose $v$ is not a coloop of $M_{A}(G(v,\ell))$. The
discussion proceeds as in case 1, with $G$ and $G^{v}$ interchanged.

\textbf{Case 3}. Suppose $v$ is a coloop of both $M_{A}(G(v,\ell))$ and
$M_{A}(G^{v}(v,\ell))$; then parts 1 and 2 of Theorem \ref{lcmatsharp} tell us
that $M_{A}(G(v))=M_{A}(G^{v}(v))$, $M_{A}(G(v,\ell))=M_{A}(G^{v}(v,\ell))$,
and $v$ is not a triple coloop of either $M_{A}(G(v,\ell))$ or $M_{A}%
(G^{v}(v,\ell))$. Then Theorem \ref{trio} and Corollary \ref{triple} tell us
that%
\begin{align*}
Z(M_{A}(G(v,\ell)))  &  =Z(M_{A}(G^{v}(v,\ell)))=Z(M_{A}(G(v,\ell
i)))=Z(M_{A}(G^{v}(v,\ell i)))\\
&  \subset Z(M_{A}(G(v)))=Z(M_{A}(G^{v}(v)))\text{,}%
\end{align*}
with the dimension of the larger subspace 1 more than the dimension of the
smaller. The equality $M_{A}(G(v,\ell i))=(M_{A}(G(v))-v)\oplus U_{1,1}%
(\{v\})$ follows from Theorem \ref{deldel} and Proposition \ref{new}.

To complete the proof of Theorem \ref{tripart}, we must verify the assertion
that if $G^{v}-v$ is simple, then $v$ cannot fall under case 1 of the
tripartition. Suppose $v\in V(G)$ falls under case 1, and let
\[
\mathcal{A}(G^{v}(v,\ell))=%
\begin{pmatrix}
1 & \mathbf{1} & \mathbf{0}\\
\mathbf{1} & A & B\\
\mathbf{0} & C & D
\end{pmatrix}
,
\]
with the first row and column corresponding to $v$. As $v$ is not a coloop of
$M_{A}(G^{v}(v,\ell))$, the first column of $\mathcal{A}(G^{v}(v,\ell))$ must
equal the sum of the columns corresponding to elements of some subset
$T\subseteq V(G)\backslash\{v\}$. Consider the submatrix of $\mathcal{A}%
(G^{v}(v,\ell))$ obtained by removing the rows and columns corresponding to
vertices not in $T\cup\{v\}$,
\[
\mathcal{A}(G^{v}(v,\ell)[T\cup\{v\}])=%
\begin{pmatrix}
1 & \mathbf{1} & \mathbf{0}\\
\mathbf{1} & A^{\prime} & B^{\prime}\\
\mathbf{0} & C^{\prime} & D^{\prime}%
\end{pmatrix}
.
\]
The sum of the columns of this matrix is $\mathbf{0}$, so the sum of the
entries of the matrix is $0$; that is, the matrix has an even number of
nonzero entries. As the matrix is symmetric, an even number of these nonzero
entries occur off the diagonal; consequently an even number occur on the
diagonal, so at least one element of $T$ is looped in $G^{v}$.

Essentially the same argument proves that in case 2, at least one element of
$T$ is looped in $G$. We should point out that a garbled version of this
simple argument appeared in \cite{Tv}, where it was mistakenly understood to
imply that there must be at least one looped vertex in $T\cap N(v)$. This need
not be the case, as indicated by the third example in the next section. The
statements of Lemma 4.4 and Corollary 4.6 of \cite{Tv} should be corrected by
replacing the hypothesis \textquotedblleft if $a$ has no looped
neighbor\textquotedblright\ with \textquotedblleft if $H-a$ has no looped
vertex.\textquotedblright

\section{Three examples}

Recall that if $k<n$ then $U_{n,k}$ denotes the $n$-element matroid whose
circuits include all the $(k+1)$-element subsets of the ground set. Also,
$U_{n,n}$ denotes the free matroid on $n$ elements, i.e., $\mathcal{C}%
(U_{n,n})$ $=$ $\varnothing$.

Let $K_{3}$ be the complete graph with three vertices. Then $M_{A}(K_{3})\cong
U_{3,2}$. If $v\in V(K_{3})$ then $v$ is not a coloop of either $M_{A}(K_{3})$
or $M_{A}(K_{3}^{v})$; $v$ falls under case 3 of the principal vertex
tripartition. $M_{A}(K_{3}(v,\ell))$ $=$ $M_{A}(K_{3}(v,\ell i))\cong U_{3,3}%
$, $M_{A}(K_{3}^{v})$ $=$ $M_{A}(K_{3})$, $M_{A}(K_{3}-v)$ $=$ $M_{A}%
(K_{3})-v\cong U_{2,2}$, and $M_{A}(K_{3})/v\cong U_{2,1}$.

Let $K_{3\ell}$ be the graph obtained from $K_{3}$ by attaching a loop to one
vertex. Then $M_{A}(K_{3\ell})\cong U_{3,3}$. If $v$ is one of the unlooped
vertices then $v$ is a coloop of $M_{A}(K_{3\ell})$, and $v$ is a triple
coloop of $M_{A}(K_{3\ell}^{v})$; $v$ falls under case 2 of the principal
vertex tripartition. $M_{A}(K_{3\ell}(v,\ell i))$ $=$ $M_{A}(K_{3\ell})$,
$M_{A}(K_{3\ell}(v,\ell))\cong U_{1,1}\oplus U_{2,1}$, and $M_{A}(K_{3\ell
}-v)$ $=$ $M_{A}(K_{3\ell})-v=M_{A}(K_{3\ell})/v\cong U_{2,2}$. If $w$ is the
looped vertex then $w$ is a coloop of $M_{A}(K_{3\ell})$ and a coloop of
$M_{A}(K_{3\ell}^{w})$, but not a triple coloop of either; $w$ falls under
case 3 of the principal vertex tripartition. $M_{A}(K_{3\ell}^{w})$ $=$
$M_{A}(K_{3\ell}(w,\ell i))$ $=$ $M_{A}(K_{3\ell})$, $M_{A}(K_{3\ell}(w))\cong
U_{3,2}$, and $M_{A}(K_{3\ell})-w$ $=$ $M_{A}(K_{3\ell})/w$ $=$ $M_{A}%
(K_{3\ell}-w)\cong U_{2,2}$.%

\begin{figure}
[pt]
\begin{center}
\includegraphics[
trim=1.433155in 8.321118in 1.433980in 1.148847in,
height=0.8864in,
width=3.7931in
]%
{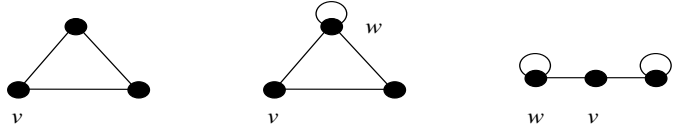}%
\caption{$K_{3}$, $K_{3\ell}$ and $P_{3\ell\ell}$.}%
\label{bingraf3}%
\end{center}
\end{figure}

Let $P_{3}$ be the path of length three, and $P_{3\ell\ell}$ the graph
obtained from $P_{3}$ by attaching loops at the vertices of degree 1.
(Equivalently, $P_{3\ell\ell}=K_{3}^{v}$.) Then $M_{A}(P_{3\ell\ell})\cong
U_{3,2}$. If $v$ is the unlooped vertex then $v$ is not a coloop of
$M_{A}(P_{3\ell\ell})$ or $M_{A}(P_{3\ell\ell}^{v})$; $v$ falls under case 3
of the principal vertex tripartition. $M_{A}(P_{3\ell\ell}(v,\ell))$ $=$
$M_{A}(P_{3\ell\ell}(v,\ell i))$ $\cong$ $U_{3,3}$, $M_{A}(P_{3\ell\ell}^{v})$
$=$ $M_{A}(P_{3\ell\ell})$, $M_{A}(P_{3\ell\ell}-v)$ $=$ $M_{A}(P_{3\ell\ell
})-v$ $\cong$ $U_{2,2}$ and $M_{A}(P_{3\ell\ell})/v$ $\cong$ $U_{2,1}$. If $w$
is one of the looped vertices then $w$ is not a coloop of $M_{A}(P_{3\ell\ell
})$, and $w$ is a triple coloop of $M_{A}(P_{3\ell\ell}^{w})$; $w$ falls under
case 2 of the principal vertex tripartition. $M_{A}(P_{3\ell\ell}(w))$ $=$
$M_{A}(P_{3\ell\ell}(w,\ell i))$ $=$ $M_{A}(P_{3\ell\ell}^{w})$ $\cong$
$U_{3,3}$, $M_{A}(P_{3\ell\ell}-w)$ $=$ $M_{A}(P_{3\ell\ell})-w\cong$
$U_{2,2}$, and $M_{A}(P_{3\ell\ell})/w\cong U_{2,1}$.

Observe that $K_{3}$ and $P_{3\ell\ell}$ have isomorphic adjacency matroids,
and their principal vertex tripartitions are distinct. On the other hand,
$K_{3\ell}$ and $P_{3\ell\ell}$ have nonisomorphic adjacency matroids and
equivalent principal vertex tripartitions. We deduce Theorem \ref{ind} of the introduction.

Note also that both $U_{3,2}$ and $U_{3,3}$ have the property that every
permutation of the ground set is a matroid automorphism. Consequently for each
of these matroids, all elements of the ground set are equivalent under the
principal edge tripartition of Rosenstiehl and Read \cite{RR}. In contrast,
the elements of $M_{A}(K_{3\ell})$ and $M_{A}(P_{3\ell\ell})$ are not
equivalent under the principal vertex tripartition.

\section{Set systems and $\Delta$-matroids}

\label{sec:recall_ss_dm} In this section we briefly summarize a number of
definitions and results related to set systems and $\Delta$-matroids. We refer
to \cite{Bu1} -- \cite{BH3} for detailed discussions.

\subsection{Set systems}

A \emph{set system} (over $V$) is a tuple $D=(V,\sigma)$ with $V$ a finite
set, called the \emph{ground set}, and $\sigma$ a family of subsets of $V$. We
often write $Y\in D$ to mean $Y\in\sigma$. A set system $D$ is called
\emph{proper} if $\sigma\not =\varnothing$, and \emph{normal} if
$\varnothing\in D$. Let $X\subseteq V$. If $D$ is proper, then we define the
\emph{distance} between $X\subseteq V$ and $D$ by $d_{D}(X)=\min(\{|X\Delta
Y|$ $\mid$ $Y\in D\})$. Moreover, we let $d_{D}=d_{D}(\varnothing)$, so that
$D$ is normal if and only if $d_{D}=0$. We define the \emph{restriction} of
$D$ to $X$ by $D[X]=(X,\sigma^{\prime})$ where $\sigma^{\prime}=\{Y\in
\sigma\mid Y\subseteq X\}$, and the \emph{deletion} of $X$ from $D$ by
$D-X=D[V-X]$. Let $\min(\sigma)$ ($\max(\sigma)$, resp.) denote the family of
minimal (maximal, resp.) sets in $\sigma$ with respect to set inclusion, and
let $\min(D)=(V,\min(\sigma))$ ($\max(D)=(V,\max(\sigma))$, resp.) be the
corresponding set systems. A set system $D$ is \emph{equicardinal} if for all
$X_{1},X_{2}\in D$, $|X_{1}|=|X_{2}|$.

Let $D$ again be a set system. For $X\subseteq V$ we define the \emph{pivot}
(also called \emph{twist} in the literature \cite{Bu1}) by $D\ast
X=(V,\sigma\ast X)$, where $\sigma\ast X=\{Y\Delta X$ $\mid$ $Y\in\sigma\}$.
Also, if $v\in V$ we define the \emph{contraction} of $D/v$ by $D/v=D\ast v-v$
\cite[Property~2.1]{Bu3}. Note that $d_{D}(X)=d_{D\ast X}$; in particular,
$D\ast X$ is normal if and only if $X\in D$. Also, note that $D\ast V$ is
obtained from $D$ by complementing every set of $D$ with respect to the ground
set. Thus, it is easy to see that $\min(D)=\max(D\ast V)\ast V$ and
$\max(D)=\min(D\ast V)\ast V$. For $X\subseteq V$ we define \emph{loop
complementation} by $D+X=(V,\sigma^{\prime})$, where $Y\in\sigma^{\prime}$ iff
$|\{Z\in D\mid(Y\setminus X)\subseteq Z\subseteq Y\}|$ is odd \cite{BH1}. In
particular, if $v\in V$ then $D+v=(V,\sigma^{\prime})$ with $\sigma^{\prime
}=\{Y\mid v\notin Y\in D\}\cup\{Y\cup\{v\}$ $\mid$ $Y\in D$ and $Y\cup
\{v\}\notin D\}$. For $X\subseteq V$ we define the \emph{dual pivot} by
$D\bar{\ast}X=D+X\ast X+X$. It turns out that $D\bar{\ast}X=(V,\sigma^{\prime
})$, where $Y\in\sigma^{\prime}$ iff $|\{Z\in D\mid Y\subseteq Z\subseteq
Y\cup X\}|$ is odd. It is easy to verify that $\max(D)=\max(D\bar{\ast}X)$ for
all $X\subseteq V$.

For convenience, we often write $D-\{v\}$, $D\ast\{v\}$, $D\bar{\ast}\{v\}$
etc. simply as $D-v$, $D\ast v$, $D\bar{\ast}v$ etc. Also, we assume
left-associativity of the operations. E.g., $D\bar{\ast}v\ast w-v$ denotes
$((D\bar{\ast}v)\ast w)-v$. Deletion, pivot, loop complementation, and dual
pivot belong to a class of operations called vertex flips which commute on
different elements (see \cite{BH1}). For example, for $v,w\in V$ and
$v\not =w$, $D-v\ast w=D\ast w-v$, $D\bar{\ast}v\ast w=D\ast w\bar{\ast}v$ and
$D\bar{\ast}v\bar{\ast}w=D\bar{\ast}w\bar{\ast}v$. Moreover, pivot, loop
complementation, and dual pivot are involutions.

Suppose $D=(V,\sigma)$ is a set system, and $v\in V$. Then $\sigma
=\sigma^{\prime}\cup\sigma^{\prime\prime}$, where $Y\in\sigma^{\prime}$ (resp.
$Y\in\sigma^{\prime\prime})$ iff $v\not \in Y\in\sigma$ (resp. $v\in
Y\in\sigma$). We define $D\widetilde{-}v=(V,\sigma^{\prime})$ and
$D\widetilde{/}v=(V,\sigma^{\prime\prime})$. That is, $D\widetilde{-}v$ is the
set system on $V$ that includes the same sets as $D-v$, and $D\widetilde{/}v$
is the set system on $V$ that includes the sets $Y\cup\{v\}$ with $Y\in D/v$.

\begin{theorem}
\label{thm:ssil} Let $D$ be a set system on $V$, and suppose $v\in V$ has the
property that $D\widetilde{-}v$ is a proper set system. We have
\[
\max(D\widetilde{-}v+v)=(\max(D\ast v))\widetilde{/}v.
\]
That is, $\max(D\widetilde{-}v+v)$ is obtained from $\max(D\ast v)$ by
removing the sets that do not contain $v$.
\end{theorem}

\begin{proof}
By definition, $D\widetilde{-}v+v=(V,\sigma)$ where $\sigma=\{Y,Y\cup\{v\}\mid
v\notin Y\in D\}$. Consequently all sets in $\max(D\widetilde{-}v+v)$ contain
$v$; that is, $\max(D\widetilde{-}v+v)=\max((D\widetilde{-}v+v)\widetilde
{/}v)$. Moreover, the family of sets of $D\widetilde{-}v+v$ that include $v$
is $\{Y\cup\{v\}\mid$ $v\notin Y\in D\}$, which is equal to the family of sets
of $D\ast v$ that include $v$. That is, $(D\widetilde{-}v+v)\widetilde
{/}v=(D\ast v)\widetilde{/}v$; it follows immediately that $\max
((D\widetilde{-}v+v)\widetilde{/}v)=\max((D\ast v)\widetilde{/}v)$. The
equality $(\max((D\ast v))\widetilde{/}v=\max((D\ast v)\widetilde{/}v)$ is
obvious; the maximal elements of $D\ast v$ that contain $v$ are the elements
of $D\ast v$ that are maximal among those that contain $v$.
\end{proof}

Theorem \ref{thm:ssil} is the first of several results that extend properties
of adjacency matroids to general set systems or $\Delta$-matroids. As we
discuss in Theorem \ref{new1} below, this theorem extends part of Proposition
\ref{new} of the introduction.

\subsection{$\Delta$-matroids}

As mentioned in the introduction, a \emph{delta-matroid} ($\Delta$-matroid for
short) is a proper set system $D$ that satisfies the \emph{symmetric exchange
axiom}: For all $X,Y\in D$ and all $u\in X\Delta Y$, $X\Delta\{u\}\in D$ or
there is a $v\in X\Delta Y$ with $v\not =u$ such that $X\Delta\{u,v\}\in D$
(or both) \cite{Bu1}. A proper set system $D$ is a $\Delta$-matroid if and
only if for each $X\subseteq V$, $\min(D\ast X)$ is equicardinal (see
\cite{BH2}). Equivalently, $D$ is a $\Delta$-matroid if and only if for each
$X\subseteq V$, $\max(D\ast X)$ is equicardinal. Let $D$ be a $\Delta$-matroid
and $v\in V$. Then $D\ast v$ is a $\Delta$-matroid. Moreover, $D-v$ is a
$\Delta$-matroid if and only if $D-v$ is proper.

However $D\bar{\ast}v$ may be a proper set system without being a $\Delta
$-matroid. As in \cite[Example 10]{BH1}, let $V$ be a finite set with
$\left\vert V\right\vert \geq3$, and consider the $\Delta$-matroid
$D=(V,\sigma)$ with $\sigma=2^{V}\backslash\{\varnothing\}$. Then it is easy
to see that the symmetric exchange axiom does not hold for $D\bar{\ast
}V=(V,\{\varnothing,V\})$.

If we assume a matroid $M$ is described by its family of bases, i.e., $M$ is
the set system $(V,B)$ where $B$ is the set of bases of $M$, then it is shown
in \cite[Proposition~3]{Bu2} that a matroid $M$ is precisely an equicardinal
$\Delta$-matroid. Moreover, a proper set system $D$ is a $\Delta$-matroid if
and only if for each $X\subseteq V$, $\max(D\ast X)$ is a matroid
\cite[Property~4.1]{Bu3}. Note that for a matroid $M$ (described by its family
of bases), $M\ast V$ is the dual matroid of $M$. Hence, $D$ is a $\Delta
$-matroid if and only if for each $X\subseteq V$, $\min(D\ast X)$ is a
matroid. Clearly for any $\Delta$-matroid $D$, $r(\min(D))=d_{D}$ and
$\nu(\max(D))=d_{D}(V)$, where $r$ and $\nu$ denote the rank and nullity of a
matroid respectively. The deletion operation of $\Delta$-matroids coincides
with the deletion operation of matroids only for non-coloops. Also, the
contraction operation of $\Delta$-matroids coincides with the contraction
operation of matroids only for non-loops. Fortunately, as deletion and
contraction for matroids coincide for both loops and coloops, matroid-deletion
of a coloop is equal to $\Delta$-matroid-contraction of that element, and
matroid-contraction of a loop is equal to $\Delta$-matroid-deletion of that element.

We will need Theorem 14 from \cite{BH2} (the original formulation is in terms
of rank rather than nullity).

\begin{proposition}
\label{prop:dist_char_min} Let $D$ be a $\Delta$-matroid, and suppose $v\in V$
has the property that $D+v$ is also a $\Delta$-matroid. Then $\max(D)$,
$\max(D\ast v)$, and $\max(D+v)$ are matroids such that precisely two of the
three are equal, to say $D_{1}$. Moreover the third, $D_{2}$, has
$(D_{2}-v)\oplus U_{1,1}(\{v\})=D_{1}$ and $\nu(D_{2})=\nu(D_{1})+1$.
\end{proposition}

Note that consequently, $\nu(\max(D))=\nu(\max(D\ast v))$ if and only if
$\max(D)=\max(D\ast v)$. Also, Theorem \ref{thm:ssil} tells us that if
$D\widetilde{-}v$ is proper, then $\max(D\widetilde{-}v+v)$ can replace
$\max(D\ast v)$ in Proposition \ref{prop:dist_char_min}: if $\max(D\ast
v)=D_{1}$ then $v$ is a coloop of $\max(D\ast v)=(\max(D\ast v))\widetilde
{/}v=\max(D\widetilde{-}v+v)$, and if $\max(D\ast v)=D_{2}$ then $\max(D\ast
v)$ and $(\max(D\ast v))\widetilde{/}v=(\max(D\ast v)/v)\oplus U_{1,1}(\{v\})$
are different matroids ($v$ is a coloop in the latter but not the former) with
the same nullity.

\subsection{Representing graphs by $\Delta$-matroids}

Let $G=(V,E)$ be a graph. Recall\ Definition \ref{gmatroid}: $\mathcal{D}_{G}$
is the set system $(V,\sigma)$ where $\sigma=\{X\subseteq V\mid\mathcal{A}%
(G)[X]$ is nonsingular over $GF(2)\}$. It is shown in \cite{Bu1} that
$\mathcal{D}_{G}$ is a normal $\Delta$-matroid (by convention, the empty
matrix is nonsingular). Moreover, if $G$ is a looped simple graph then given
$\mathcal{D}_{G}$, one can (re)construct $G$: $\{u\}$ is a loop in $G$ if and
only if $\{u\}\in\mathcal{D}_{G}$, and $\{u,v\}$ is an edge in $G$ if and only
if $(\{u,v\}\in\mathcal{D}_{G})\Delta((\{u\}\in\mathcal{D}_{G})\wedge
(\{v\}\in\mathcal{D}_{G}))$, see \cite[Property~3.1]{Bu3}. In this way, the
family of looped simple graphs with vertex-set $V$ can be considered as a
subset of the family of $\Delta$-matroids on the ground set $V$.

It is shown in \cite[Theorem~15]{BH2} that, for all $X\subseteq V$,
$d_{\mathcal{D}_{G}}(X)=\nu(\mathcal{A}(G)[X])$, where $\nu$ denotes
$GF(2)$-nullity. If $v$ is a looped vertex of $G$, then it is shown in
\cite{G} that $\mathcal{D}_{G}\ast v$ represents the graph $G^{v}$. Moreover,
if $v$ is a unlooped vertex of $G$, then $\mathcal{D}_{G}\bar{\ast}v$
represents the graph $G^{v}$ (see \cite{BH1}). In this way, $G^{v}$ may be
defined using $\Delta$-matroids. However, $\mathcal{D}_{G}\ast v$ on an
unlooped vertex $v$ and $\mathcal{D}_{G}\bar{\ast}v$ on a looped vertex $v$ do
\emph{not} represent graphs in general:

\begin{proposition}
[Remark below Theorem~22 in \cite{BH2}]\label{prop:vf_close_to_graph} Let $G$
be a graph, and $\varphi$ be any sequence of pivot, dual pivot and loop
complement operations on elements of $V(G)$. Then $\varnothing\in
(\mathcal{D}_{G})\varphi$ if and only if $(\mathcal{D}_{G})\varphi
=\mathcal{D}_{G^{\prime}}$ for some graph $G^{\prime}$.
\end{proposition}

In contrast, $\max((\mathcal{D}_{G})\varphi)$ \emph{does} always have a graph representation.

\begin{theorem}
Let $G$ be a graph, and let $\varphi$ be any sequence of pivot, dual pivot and
local complement operations on elements of $V(G)$. Then $\max((\mathcal{D}%
_{G})\varphi)$ is a binary matroid.
\end{theorem}

\begin{proof}
Let $D=(\mathcal{D}_{G})\varphi$. Let $X\in\min(D)$. Then $\varnothing\in
D\bar{\ast}X$ if and only if $|\{Z\in D\mid Z\subseteq X\}|$ is odd. Since
$\{Z\in D\mid Z\subseteq X\}=\{X\}$ by definition of $X$, we have
$\varnothing\in D\bar{\ast}X=(\mathcal{D}_{G})\varphi\bar{\ast}X$. By
Proposition~\ref{prop:vf_close_to_graph}, $(\mathcal{D}_{G})\varphi\bar{\ast
}X=\mathcal{D}_{G^{\prime}}$ for some graph $G^{\prime}$. Thus, $\max
((\mathcal{D}_{G})\varphi)=\max((\mathcal{D}_{G})\varphi\bar{\ast}%
X)=\max(\mathcal{D}_{G^{\prime}})$ and we are done.
\end{proof}

For convenience, we define the \emph{pivot} of a vertex $v$ on a graph $G$,
denoted $G\ast v$, by $G^{v}$ if $v$ is looped, and it is not defined
otherwise. Similarly, we define the \emph{dual pivot} of vertex $v$ on $G$,
denoted $G\bar{\ast}v$, by $G^{v}$ if $v$ is unlooped, and it is not defined
otherwise. For a graph, \emph{loop complementation} of a vertex $v\in V(G)$,
denoted by $G+v$, toggles the existence of a loop on $v$. I.e., $v$ is a
looped vertex of $G$ iff $v$ is not a looped vertex of $G+v$. It is shown in
\cite{BH1} that $\mathcal{D}_{G+v}=\mathcal{D}_{G}+v$ (i.e., loop
complementation for $\Delta$-matroids generalizes loop complementation for graphs).

It is easy to verify that for each $v\in V(G)$, $\mathcal{D}_{G-v}%
=\mathcal{D}_{G}-v$. Theorem \ref{matdmat} of the introduction follows readily
from this easy observation and the strong principal minor theorem \cite{K}
(see also Theorem \ref{spmt} above). As we will see in the following sections,
the equality%
\[
\max(\mathcal{D}_{G})=M_{A}(G)
\]
(where $M_{A}(G)$ is described by its family of bases) allows us to give
various results stated in the introduction completely different proofs, using
$\Delta$-matroids rather than linear algebra over $GF(2)$.

\section{Deletion/contraction and $\min$/$\max$ for $\Delta$-matroids}

\label{sec:del_con_minmax} In this section we show, under the assumption of
some mild conditions, that both contraction and deletion commute with both the
$\min$ and the $\max$ operation for $\Delta$-matroids. In fact, some results
hold for set systems in general. We will apply these results to graphs in the
next section.

Let $D$ be a set system. The notions of loop and coloop for matroids
(described by their families of bases) may be directly generalized to set
systems. An element $v\in V$ is called a \emph{coloop} of $D$ if $v\in X$ for
each $X\in D$. Clearly, $v$ is a coloop of $D$ if and only if $D-v$ is not
proper. Similarly, $v\in V$ is called a \emph{loop} of $D$ if $v$ is a coloop
of $D\ast v$, i.e., $v\not \in X$ for each $X\in D$.

We first show that the $\min$ operation and the deletion operation on an
element $v$ commute for proper set systems $D$, provided that $D-v$ is proper.
Note that $D-v$ is proper, i.e., $v$ is not a coloop of $D$, if and only if
$v$ is not a coloop of $\min(D)$.

\begin{theorem}
\label{thm:min_delete} Let $D$ be a proper set system, and let $v \in V$ such
that $D - v$ is proper. Then $\min(D) - v = \min(D - v)$.
\end{theorem}

\begin{proof}
Since $D - v$ is proper, $\min(D - v)$ is well defined. Let $X \in\min(D) -
v$. Then $X \in\min(D)$ and $v \not \in X$. Hence, $X \in\min(D - v)$.
Conversely, if $X \in\min(D - v)$, then $X \in D$ and $v \not \in X$. Let $Y
\subseteq X$ with $Y \in\min(D)$. Then clearly, $v \not \in Y$ and thus $Y
\in\min(D - v)$. Hence $X = Y$ and $X \in\min(D)$. Therefore, $X \in\min(D) -
v$.
\end{proof}

Next, we show that the $\max$ operation and the deletion operation on a
element $v$ commute for $\Delta$-matroids $D$, provided that $v$ is not a
coloop of $\max(D)$.

\begin{theorem}
\label{thm:max_delete} Let $D$ be a $\Delta$-matroid, and let $v \in V$ such
that $v$ is not a coloop of $\max(D)$. Then $\max(D) - v = \max(D - v)$.
\end{theorem}

\begin{proof}
Since $v$ is not a coloop of $\max(D)$, there is a $Z \in\max(D)$ with $v
\not \in Z$. Therefore $D - v$ is proper, and so $\max(D - v)$ is well
defined. Let $X \in\max(D) - v$. Then $X \in\max(D)$ and $v \not \in X$.
Hence, $X \in\max(D - v)$. Conversely, if $X \in\max(D - v)$, then $X \in D$
and $v \not \in X$, and $X$ is maximal with this property. As $v$ is not a
coloop of $\max(D)$, there is a $Z \in\max(D)$ with $v \not \in Z$. Therefore,
$Z \in\max(D) - v$. By the first part of this proof, $Z \in\max(D - v)$. Now,
as $D - v$ is a $\Delta$-matroid, $\max(D - v)$ is equicardinal and so $|Z| =
|X|$. Moreover, since $D$ is a $\Delta$-matroid, $\max(D)$ is equicardinal,
therefore $X \in\max(D)$ and so, $X \in\max(D) - v$.
\end{proof}

The next example illustrates that Theorem~\ref{thm:max_delete} does not hold
for set systems in general. This in contrast with Theorem~\ref{thm:min_delete}%
, which \emph{does} hold for set systems in general.

\begin{example}
Let $D = (V,D)$ be a set system with $V = \{u,v,w\}$ and $D = \{
\{u\},\{v\},\{v,w\} \}$. Then $w$ is not a coloop of $\max(D) = (V,\{
\{u\},\{v,w\} \})$, and $\max(D) - w = (\{u,v\},\{\{u\}\})$ while $\max(D - w)
= (\{u,v\},\{\{u\},\{v\}\})$.
\end{example}

We formulate now the max (min, resp.) \textquotedblleft
counterparts\textquotedblright\ of Theorem~\ref{thm:min_delete}
(Theorem~\ref{thm:max_delete}, resp.). These results show that contraction
commutes with the $\min$ and $\max$ operations.

\begin{theorem}
\label{thm:comm_contract_minmax} Let $D$ be a proper set system and $v \in V$.

\begin{enumerate}
\item If $v$ is not a loop of $D$, then $\max(D)\ast v-v=\max(D\ast v-v)$.

\item If $D$ is moreover a $\Delta$-matroid and $v$ is not a loop of $\min
(D)$, then $\min(D)*v - v = \min(D*v - v)$.
\end{enumerate}
\end{theorem}

\begin{proof}
We start by showing the first result. We have $\max(D)\ast v-v=\min(D\ast
V)\ast V\ast v-v=\min(D\ast V)-v\ast(V\backslash\{v\})$. Now, $v$ is not a
coloop of $D\ast V$. Thus, $D\ast V-v$ is proper. By
Theorem~\ref{thm:min_delete}, $\min(D\ast V)-v\ast(V\backslash\{v\})=\min
(D\ast V-v)\ast(V\backslash\{v\})=\min(D\ast v-v\ast(V\backslash
\{v\}))\ast(V\backslash\{v\})=\max(D\ast v-v)$. The proof of the second result
is essentially identical to that of the first result. We have $\min(D)\ast
v-v=\max(D\ast V)\ast V\ast v-v=\max(D\ast V)-v\ast(V\backslash\{v\})$. Now,
$v$ is not a coloop of $\min(D)\ast V=\max(D\ast V)$. By
Theorem~\ref{thm:max_delete}, $\max(D\ast V)-v\ast(V\backslash\{v\})=\max
(D\ast V-v)\ast(V\backslash\{v\})=\max(D\ast v-v\ast(V\backslash
\{v\}))\ast(V\backslash\{v\})=\min(D\ast v-v)$.
\end{proof}

\section{From $\Delta$-matroids to graphs}

\label{sec:dmat_alt_proofs_mat} In this section we use results of Sections
\ref{sec:recall_ss_dm} and \ref{sec:del_con_minmax} to give new proofs of
several theorems about adjacency matroids stated earlier in the paper. These
proofs are fundamentally different from the earlier ones, as they are
combinatorial and do not involve matrices. Recall that if $G$ is a graph then
$\mathcal{D}_{G}$ is a $\Delta$-matroid with $M_{A}(G)=\max(\mathcal{D}_{G})$;
$\mathcal{D}_{G}$ is normal, so no $v\in V$ is a coloop of $\mathcal{D}_{G}.$

The following three results are quite straightforward consequences of the fact
that for $\mathcal{D}_{G}$, $\max$ commutes with both deletion (of
non-coloops) and contraction (of non-loops), cf. Theorems~\ref{thm:max_delete}
and \ref{thm:comm_contract_minmax}.

\begin{theorem}
\label{thm:mat_del} If $v$ is not a coloop of $M_{A}(G)$, then $M_{A}%
(G)-v=M_{A}(G-v)$.
\end{theorem}

\begin{proof}
If $v$ is not a coloop of $M_{A}(G)$, then $M_{A}(G) - v = \max(\mathcal{D}%
_{G}) - v$. By Theorem~\ref{thm:max_delete}, $\max(\mathcal{D}_{G}) - v =
\max(\mathcal{D}_{G} - v) = \max(\mathcal{D}_{G - v}) = M_{A}(G - v)$.
\end{proof}

\begin{theorem}
\label{thm:mat_contr_loop} If $v\in V(G)$ is a looped vertex, then
$M_{A}(G)/v=M_{A}(G^{v}-v)$.
\end{theorem}

\begin{proof}
Since $\{v\}\in\mathcal{D}_{G}$, $v$ is not a loop of $\mathcal{D}_{G}$.
Consequently, $v$ is not a loop of $\max(\mathcal{D}_{G})$. We have therefore
$M_{A}(G)/v=\max(\mathcal{D}_{G})\ast v-v$. By
Theorem~\ref{thm:comm_contract_minmax}, $\max(\mathcal{D}_{G})\ast
v-v=\max(\mathcal{D}_{G}\ast v-v)$. Since $v$ is a looped vertex,
$\mathcal{D}_{G}\ast v=\mathcal{D}_{G\ast v}$, and thus $\max(\mathcal{D}%
_{G}\ast v-v)=\max(\mathcal{D}_{G\ast v-v})=M_{A}(G\ast v-v)$. The result
follows as $G\ast v=G^{v}$.
\end{proof}

\begin{theorem}
\label{thm:mat_contr_unloop} Suppose $v$ is an unlooped vertex of $G$.

\begin{enumerate}
\item \label{item:one} If $v$ is isolated, then $M_{A}(G)/v=M_{A}(G-v)$.

\item \label{item:two} If $w$ is an unlooped neighbor of $v$, then
$M_{A}(G)/v=M_{A}((G^{w})^{v}-v)$.

\item \label{item:three} If $w$ is a looped neighbor of $v$, then
$M_{A}(G)/v=M_{A}(((G^{v})^{w})^{v}-v)$.
\end{enumerate}
\end{theorem}

\begin{proof}
We first prove Result~\ref{item:one}. If $v$ is isolated and unlooped, then
$v$ is a loop of $M_{A}(G)$. Hence, $M_{A}(G)/v=M_{A}(G)-v$. Moreover, $v$ is
not a coloop of $M_{A}(G)$. The result follows now by
Theorem~\ref{thm:mat_del}.

We now prove Results~\ref{item:two} and \ref{item:three}. Let $w$ be a
neighbor of $v$. As $\{v,w\}\in\mathcal{D}_{G}$, $v$ is not a loop of
$\mathcal{D}_{G}$. Hence, $M_{A}(G)/v=\max(\mathcal{D}_{G})\ast v-v$. By
Theorem~\ref{thm:comm_contract_minmax}, $\max(\mathcal{D}_{G})\ast
v-v=\max(\mathcal{D}_{G}\ast v-v)$. Now, $\max(\mathcal{D}_{G}\ast
v-v)=\max(\mathcal{D}_{G}\ast v-v\bar{\ast}w)=\max(\mathcal{D}_{G}\bar{\ast
}w\ast v-v)$. On the one hand, if $w$ is unlooped, then it is easy to verify
that $G\bar{\ast}w\ast v$ is defined. Hence $\mathcal{D}_{G}\bar{\ast}w\ast
v=\mathcal{D}_{G\bar{\ast}w\ast v}$. Finally, $\max(\mathcal{D}_{G\bar{\ast
}w\ast v}-v)=M_{A}(G\bar{\ast}w\ast v-v)=M_{A}((G^{w})^{v}-v)$. This proves
Result~\ref{item:two}. On the other hand, if $w$ is looped, then it is easy to
verify that $G\bar{\ast}v\bar{\ast}w\ast v$ is defined. Hence $\mathcal{D}%
_{G}\bar{\ast}v\bar{\ast}w\ast v=\mathcal{D}_{G\bar{\ast}v\bar{\ast}w\ast v}$.
Finally, $\max(\mathcal{D}_{G}\bar{\ast}w\ast v-v)=\max(\mathcal{D}%
_{G\bar{\ast}v\bar{\ast}w\ast v}-v)=M_{A}(G\bar{\ast}v\bar{\ast}w\ast
v-v)=M_{A}(((G^{v})^{w})^{v}-v)$. This proves Result~\ref{item:three}.
\end{proof}

The next result is obtained from Proposition~\ref{prop:dist_char_min}.

\begin{theorem}
\label{thm:eq_matroid_lc}

\begin{enumerate}
\item \label{item2:one} If $v \in V(G)$ is unlooped, then $M_{A}(G^{v}) =
M_{A}(G)$.

\item \label{item2:two} If $v\in V(G)$ is a coloop of both $M_{A}(G)$ and
$M_{A}(G^{v})$, then $M_{A}(G^{v})=M_{A}(G)$.

\item \label{item2:three} If $v\in V(G)$ is looped and not a coloop of one of
$M_{A}(G)$, $M_{A}(G^{v})$, then $v$ is a coloop of the other and $M_{A}%
(G^{v})$ and $M_{A}(G)$ are of different ranks.
\end{enumerate}
\end{theorem}

\begin{proof}
We first show Result~\ref{item2:one}. If $v$ is unlooped, then $G^{v}%
=G\bar{\ast}v$. Thus, $M_{A}(G)=\max(\mathcal{D}_{G})=\max(\mathcal{D}_{G}%
\bar{\ast}v)=\max(\mathcal{D}_{G\bar{\ast}v})=M_{A}(G^{v})$ and the result follows.

We now show Results~2 and 3. If $v$ is unlooped, then we are done by Result~1.
So, assume $v$ is looped. Then we have $G^{v}=G\ast v$. The result follows now
by Proposition~\ref{prop:dist_char_min}.
\end{proof}

\begin{theorem}
If $v\in V(G)$, then $M_{A}(G)-v=M_{A}(G^{v})-v$.
\end{theorem}

\begin{proof}
If $v$ is an unlooped vertex, then by Theorem~\ref{thm:eq_matroid_lc}%
.\ref{item2:one}, $M_{A}(G)=M_{A}(G^{v})$ and the equality holds. Assume now
that $v$ is a looped vertex. By Theorem~\ref{thm:eq_matroid_lc}%
.\ref{item2:three}, $v$ is a coloop of at least one of $M_{A}(G)$ and
$M_{A}(G^{v})$. If $v$ is a coloop of both $M_{A}(G)$ and $M_{A}(G^{v})$, then
the equality holds by Theorem~\ref{thm:eq_matroid_lc}.\ref{item2:two}.

We assume now without loss of generality that $v$ is a coloop of $M_{A}(G)$
and $v$ is not a coloop of $M_{A}(G^{v})$ (the other case follows by
considering graph $G:=G^{v}$). We have in this case $M_{A}(G)-v=M_{A}(G)/v$.
By Theorem~\ref{thm:mat_contr_loop}, $M_{A}(G)/v=M_{A}(G^{v}-v)$. As $v$ is
not a coloop of $M_{A}(G^{v})$, by Theorem~\ref{thm:mat_del}, $M_{A}%
(G^{v}-v)=M_{A}(G^{v})-v$ and the result follows.
\end{proof}

By the way, the interested reader will have no trouble using \cite[Theorem
15]{BH2} to prove Theorems \ref{lcmatsharp} and \ref{deldelsharp}.

It is easy to see that $\mathcal{D}_{G}\widetilde{-}v+v=\mathcal{D}_{G(v,li)}%
$. Hence we obtain the following corollary to Theorem~\ref{thm:ssil} and
Proposition \ref{prop:dist_char_min}. Part 1 follows from part 2, which is
part of Proposition \ref{new}; and part 3 includes some of the assertions of
Theorem \ref{tripart}.

\begin{theorem}
\label{new1}If $G$ is a graph with a looped vertex $v$, then the following hold.

\begin{enumerate}
\item $\nu(M_{A}(G^{v}))=\nu(M_{A}(G(v,\ell i)))$.

\item $\mathcal{B}(M_{A}(G(v,\ell i)))=\{B\in\mathcal{B}(M_{A}(G^{v}))\mid
v\in B\}$, or equivalently $M_{A}(G(v,\ell i))$ $=(M_{A}(G^{v})/v)\oplus
U_{1,1}(\{v\})$.

\item $M_{A}(G(v,\ell i))=M_{A}(G^{v})$ iff $\nu(M_{A}(G))\geq\nu(M_{A}%
(G^{v}))$.
\end{enumerate}
\end{theorem}

\section{The interlace and Tutte polynomials}

In this section we discuss the connection between the interlace polynomials of
a graph $G$, introduced by Arratia, Bollob\'{a}s and Sorkin \cite{A1, A2, A},
and the Tutte polynomials of the adjacency matroids of $G$ and its subgraphs.
(The Tutte polynomial is described by many authors; see \cite{Bo} and \cite{T}
for instance. Especially thorough accounts are given by Brylawski and Oxley
\cite{W} and Ellis-Monaghan and Merino \cite{D}.) In particular, we show that
the fundamental recursion of the two-variable interlace polynomial may be
derived from properties of the Tutte polynomial.

\begin{definition}
\label{defq}Let $G$ be a graph. Then the \emph{interlace polynomial} of $G$ is%
\begin{align*}
q(G)  &  =\sum_{S\subseteq V(G)}(x-1)^{\left\vert S\right\vert -\nu
(\mathcal{A}(G[S]))}(y-1)^{\nu(\mathcal{A}(G[S]))}\\
&  =\sum_{S\subseteq V(G)}(x-1)^{\left\vert S\right\vert }\cdot\left(
\frac{y-1}{x-1}\right)  ^{\nu(\mathcal{A}(G[S]))},
\end{align*}
where $\nu$ denotes $GF(2)$-nullity.
\end{definition}

Arratia, Bollob\'{a}s and Sorkin \cite{A} showed that $q(G)$ may also be
defined recursively:

1. If $v$ is a looped vertex of $G$ then $q(G)=q(G-v)$ $+$ $(x-1)q(G^{v}-v).$

2. If $v$ and $w$ are unlooped neighbors in $G$ then $q(G)$ $=$
$q(G-v)+q(((G^{v})^{w})^{v}-v)$ $+((x-1)^{2}-1)\cdot q(((G^{v})^{w})^{v}-v-w)$.

3. If $G$ consists solely of unlooped vertices then $q(G)=y^{\left\vert
V(G)\right\vert }$.

So far, our discussion of matroids has been focused on their circuits. Here
are two other basic definitions of matroid theory.

\begin{definition}
Let $M$ be a matroid on a set $V$. A subset $I\subseteq V$ is
\emph{independent} if $I$ contains no circuit of $M$. The \emph{rank} of a
subset $S\subseteq V$ is the cardinality of the largest independent set(s) in
$S$; it is denoted $r(S)$.
\end{definition}

All the notions of matroid theory can be equivalently defined from the
independent sets or the rank function, instead of the circuits. For instance,
Definitions \ref{del} and \ref{contr} are equivalent to: if $M$ is a matroid
on a set $V$ and $v\in V$ then the deletion $M-v$ and the contraction $M/v$
are the matroids on $V\backslash\{v\}$ with the rank functions $r_{M-v}(S)$
$=$ $r(S)$ and $r_{M/v}(S)$ $=$ $r(S\cup\{v\})-r(\{v\})$.

Recall that if $G$ is a graph, then the circuits of $M_{A}(G)$ are the minimal
nonempty subsets $S\subseteq V(G)$ such that the columns of $\mathcal{A}(G)$
corresponding to elements of $S$ are linearly dependent. It follows that the
rank in $M_{A}(G)$ of a subset $S\subseteq V(G)$ is simply the $GF(2)$-rank of
the $\left\vert V(G)\right\vert \times\left\vert S\right\vert $ submatrix of
$\mathcal{A}(G)$ obtained by removing the columns corresponding to vertices
not in $S$. This submatrix of $\mathcal{A}(G)$ is obtained from $\mathcal{A}%
(G[S])$ by adjoining rows corresponding to vertices not in $S$, so
$r(S)\geq\left\vert S\right\vert -\nu(\mathcal{A}(G[S]))$. The difference
between $r(S)$ and $\left\vert S\right\vert -\nu(\mathcal{A}(G[S]))$ varies
with $G$ and $S$, in general; however if $r(S)=r(M_{A}(G))$ then according to
the strong principal minor theorem (see \cite{K} or Theorem \ref{spmt}),
$r(S)$ $=$ $\left\vert S\right\vert -\nu(\mathcal{A}(G[S]))$.

\begin{definition}
Let $M$ be a matroid on a set $V$. Then the \emph{Tutte polynomial} of $M$ is%
\[
t(M)=\sum_{S\subseteq V}(x-1)^{r(V)-r(S)}(y-1)^{\left\vert S\right\vert
-r(S)}.
\]

\end{definition}

The equations $r_{M/v}(S)$ $=$ $r(S\cup\{v\})-r(\{v\})$ and $r_{M-v}(S)$ $=$
$r(S)$ imply that the Tutte polynomial may be calculated recursively using the
following steps.

1. If $v$ is a loop of $M$ then $t(M)$ $=$ $y\cdot t(M-v)$.

2. If $v$ is a coloop of $M$ then $t(M)$ $=$ $x\cdot t(M/v)$.

3. If $v$ is neither a loop nor a coloop, then $t(M)$ $=$ $t(M/v)+t(M-v)$.

4. $t(\varnothing)=1$.

\noindent The distinction between $M-v$ and $M/v$ in steps 1 and 2 is
traditional, but for us it is unimportant as Definitions\ \ref{del} and
\ref{contr} have $M/v$ $=$ $M-v$ for loops and coloops.

We single out the leading term of $t(M)$ (the term corresponding to $S$ $=$
$V$) for special attention.

\begin{definition}
\label{lambda}Let $M$ be a matroid on a set $V$. Then $\lambda_{M}(y)$ $=$
$(y-1)^{\left\vert V\right\vert -r(V)}$.
\end{definition}

Like $t(M)$, $\lambda_{M}$ has a recursive description derived from the
equations $r_{M/v}(S)$ $=$ $r(S\cup\{v\})-r(\{v\})$ and $r_{M-v}(S)$ $=$
$r(S)$:

\begin{proposition}
\label{recur}1. If $v$ is a loop of $M$ then $\lambda_{M}$ $=$ $(y-1)\cdot
\lambda_{M-v}$ $=$ $(y-1)\cdot\lambda_{M/v}$.

2. If $v$ is a coloop of $M$ then $\lambda_{M}$ $=$ $\lambda_{M-v}$ $=$
$\lambda_{M/v}$.

3. If $v$ is neither a loop nor a coloop, then $\lambda_{M}$ $=$
$(y-1)\cdot\lambda_{M-v}$ $=$ $\lambda_{M/v}$.

4. $\lambda_{\varnothing}=1$.
\end{proposition}

If $G$ is a graph then we adopt the abbreviated notation $\lambda_{M_{A}(G)} $
$=$ $\lambda_{G}$.

\begin{corollary}
\label{recloop}1. If $v$ is an unlooped vertex of $G$ then $\lambda
_{G}=\lambda_{G^{v}}$.

2. If $v$ is a looped vertex of $G$ then\ $\lambda_{G}=\lambda_{G^{v}-v}$.

3. If $v$ is an isolated, unlooped vertex of $G$ then $\lambda_{G}%
=(y-1)\cdot\lambda_{G-v}$.

4. $\lambda_{\varnothing}=1$.
\end{corollary}

\begin{proof}
If $v$ is an unlooped vertex of $G$ then $M_{A}(G)=M_{A}(G^{v})$ by Theorem
\ref{lcmat}. If $v$ is a looped vertex of $G$ then $v$ is not a loop of
$M_{A}(G)$, so $\lambda_{M_{A}(G)}$ $=$ $\lambda_{M_{A}(G)/v}$ $=$
$\lambda_{G^{v}-v}$ by Proposition \ref{recur} and Theorem \ref{loopcontr}. If
$v$ is an isolated, unlooped vertex of $G$ then $v$ is a loop of $M_{A}(G)$,
so $\lambda_{M_{A}(G)}$ $=$ $(y-1)\cdot\lambda_{M_{A}(G)-v}$ $=$
$(y-1)\cdot\lambda_{G-v}$ by Theorem \ref{deldel} and Proposition \ref{recur}.
\end{proof}

If $G$ is a graph then Definitions \ref{defq} and \ref{lambda} imply that%
\begin{equation}
q(G)=\sum_{S\subseteq V(G)}(x-1)^{\left\vert S\right\vert }\cdot\lambda
_{G[S]}(1+\frac{y-1}{x-1}) \label{eq1}%
\end{equation}
and hence for each $v\in V(G)$,%
\begin{equation}
q(G)-q(G-v)=\sum_{v\in S\subseteq V(G)}(x-1)^{\left\vert S\right\vert }%
\cdot\lambda_{G[S]}(1+\frac{y-1}{x-1}). \label{eq2}%
\end{equation}

Suppose $v$ is a looped vertex and $v\in S\subseteq V(G)$. Then Corollary
\ref{recloop} tells us that $\lambda_{G[S]}$ $=$ $\lambda_{G[S]^{v}-v}$.
Clearly $G[S]^{v}-v$ $=$ $G^{v}[S]-v$ $=$ $(G^{v}-v)[S\backslash\{v\}]$, so it
follows from equations (\ref{eq1}) and (\ref{eq2}) that%
\begin{align*}
q(G)-q(G-v)  &  =(x-1)\cdot\sum_{v\in S\subseteq V(G)}(x-1)^{\left\vert
S\right\vert -1}\cdot\lambda_{G^{v}[S]-v}(1+\frac{y-1}{x-1})\\
&  =(x-1)\cdot\sum_{S\subseteq V(G)\backslash\{v\}}(x-1)^{\left\vert
S\right\vert }\cdot\lambda_{(G^{v}-v)[S]}(1+\frac{y-1}{x-1})\\
&  =(x-1)\cdot q(G^{v}-v).
\end{align*}
This yields the first formula of the recursive description of\ $q$.

Suppose now that $v$ is an unlooped vertex of $G$, and $w$ is an unlooped
neighbor of $v$ in $G$. Let $H$ $=$ $(G^{v})^{w}$; then $v$ and $w$ are looped
neighbors in $H$. Equation (\ref{eq2}) tells us that%
\[
q(H^{v}-v)-q(H^{v}-v-w)=\sum_{w\in S\subseteq V(H^{v}-v)}(x-1)^{\left\vert
S\right\vert }\cdot\lambda_{H^{v}[S]}(1+\frac{y-1}{x-1}).
\]

Suppose $w\in S\subseteq V(H^{v}-v)$; obviously then $H^{v}[S]$ $=$
$H^{v}[S\cup\{v\}]-v$ $=$ $H[S\cup\{v\}]^{v}-v$. As $v$ and $w$ are both
looped in $H[S\cup\{v\}]$ $=$ $(G^{v})^{w}[S\cup\{v\}]$, Corollary
\ref{recloop} tells us that
\[
\lambda_{H^{v}[S]}=\lambda_{H[S\cup\{v\}]^{v}-v}=\lambda_{H[S\cup
\{v\}]}=\lambda_{(G^{v})^{w}[S\cup\{v\}]}=\lambda_{((G^{v})^{w}[S\cup
\{v\}])^{w}-w}.
\]
Note that $((G^{v})^{w}[S\cup\{v\}])^{w}-w$ $=$ $((G^{v})^{w})^{w}%
[S\cup\{v\}]-w$ $=$ $G^{v}[(S\backslash\{w\})\cup\{v\}]$ $=$ $G[(S\backslash
\{w\})\cup\{v\}]^{v}$. As $v$ is unlooped in $G[(S\backslash\{w\})\cup\{v\}]$,
Corollary \ref{recloop} tells us that
\[
\lambda_{((G^{v})^{w}[S\cup\{v\}])^{w}-w}=\lambda_{G[(S\backslash
\{w\})\cup\{v\}]^{v}}=\lambda_{G[(S\backslash\{w\})\cup\{v\}]}.
\]
We conclude that
\begin{align*}
&  q(H^{v}-v)-q(H^{v}-v-w)\\
&  =\sum_{w\in S\subseteq V(H-v)}(x-1)^{\left\vert S\right\vert }\cdot
\lambda_{G[(S\backslash\{w\})\cup\{v\}]}(1+\frac{y-1}{x-1})\\
&  =\sum_{v\in S\subseteq V(G-w)}(x-1)^{\left\vert S\right\vert }\cdot
\lambda_{G[S]}(1+\frac{y-1}{x-1}).
\end{align*}
Combining this with equation (\ref{eq2}), we see that%
\begin{align*}
&  q(G)-q(G-v)\\
&  =q(H^{v}-v)-q(H^{v}-v-w)+\sum_{v,w\in S\subseteq V(G)}(x-1)^{\left\vert
S\right\vert }\cdot\lambda_{G[S]}(1+\frac{y-1}{x-1}).
\end{align*}

Suppose now that $v,w\in S\subseteq V(G)$. Then Proposition \ref{recur} and
Theorem \ref{nonloopcontr} imply that
\[
\lambda_{G[S]}=\lambda_{M_{A}(G[S])/w}=\lambda_{(G[S]^{v})^{w}-w}%
=\lambda_{(G^{v})^{w}[S]-w}=\lambda_{H[S\backslash\{w\}]}.
\]
As $v$ is looped in $H$, Corollary \ref{recloop} states that
\[
\lambda_{H[S\backslash\{w\}]}=\lambda_{H[S\backslash\{w\}]^{v}-v}%
=\lambda_{H^{v}[S\backslash\{v,w\}]}.
\]
We conclude that
\begin{align*}
&  q(G)-q(G-v)\\
&  =q(H^{v}-v)-q(H^{v}-v-w)+\sum_{v,w\in S\subseteq V(G)}(x-1)^{\left\vert
S\right\vert }\cdot\lambda_{G[S]}(1+\frac{y-1}{x-1})\\
&  =q(H^{v}-v)-q(H^{v}-v-w)+\sum_{v,w\in S\subseteq V(G)}(x-1)^{\left\vert
S\right\vert }\cdot\lambda_{H^{v}[S\backslash\{v,w\}]}(1+\frac{y-1}{x-1})\\
&  =q(H^{v}-v)-q(H^{v}-v-w)+\sum_{S\subseteq V(H^{v}-v-w)}(x-1)^{\left\vert
S\right\vert +2}\cdot\lambda_{H^{v}[S]}(1+\frac{y-1}{x-1})\\
&  =q(H^{v}-v)-q(H^{v}-v-w)+(x-1)^{2}\cdot q(H^{v}-v-w)\text{.}%
\end{align*}
\noindent This yields the second formula of the recursive description of\ $q$.

In the years since Arratia, Bollob\'{a}s and Sorkin introduced the interlace
polynomials \cite{A1, A2, A}, several related graph polynomials have been
studied by other researchers \cite{AH, Ci, T4}. These related polynomials have
definitions similar to Definition \ref{defq}, as sums involving $GF(2)$%
-nullities of symmetric matrices. Consequently they have similar connections
with the leading terms of Tutte polynomials of adjacency matroids.

\end{document}